\pdfoutput=1

\documentclass[preprint,11pt]{elsarticle}

\usepackage{amsmath}
\usepackage{amsthm}
\usepackage{amssymb}
\usepackage{amscd}
\usepackage{amsfonts}
\usepackage{amsbsy}
\usepackage{graphicx}
\usepackage{array}
\usepackage{color}
\usepackage{epsfig}
\usepackage{url}
\usepackage{overpic}
\usepackage{epstopdf}
\usepackage{setspace}
\usepackage{tikz}
\usepackage{color}

\usepackage{float}
\usepackage{booktabs}
\usepackage{tabularx}
\usepackage{epsfig}

\usepackage[lmargin=2.5cm,tmargin=2.5cm,bmargin=2.5cm,rmargin=2.5cm]{geometry}
\usepackage{indentfirst}
\usepackage[hidelinks]{hyperref}
\hypersetup{colorlinks=true,linkcolor=red,citecolor=blue,urlcolor=blue}
\usepackage{multicol}
\usepackage[IT,hang,FIGBOTCAP,TABBOTCAP]{subfigure}

\newcommand{\CLperp}{\mathfrak{X}_{\text{\scriptsize T}}}
\newcommand{\CLequi}{\mathfrak{X}_{\text{\scriptsize S}}}
\DeclareMathOperator{\Fix}{Fix}

\usepackage{amssymb}

\newtheorem{theorem}{Theorem}
\newtheorem{definition}{Definition}
\newtheorem{proposition}{Proposition}
\newtheorem{corollary}{Corollary}
\newtheorem{lemma}{Lemma}
\theoremstyle{definition}

\newtheorem{remark}{Remark}

\newcommand{\CL}{\mathfrak{X}_{\text{\scriptsize$L$}}}

\journal{arXiv}

\begin{document}

\begin{frontmatter}



\title{Symmetric Limit Cycles in 3D Piecewise Linear Systems with Visible-visible Two-Fold Singularity}

\author[mymainaddress]{Samuel Carlos S. Ferreira \corref{cor1}}
\ead{ferreira.samuelcarlos@gmail.com}

\cortext[cor1]{Corresponding author}
\author[mymainaddress]{Bruno R. Freitas}
\ead{freitasmat@ufg.br}

\author[mymainaddress_sp]{Jo\~ao Carlos R. Medrado}
\ead{j.medrado@unesp.br}

\address[mymainaddress]{Institute of Mathematics and Statistics of Federal University of Goi\'{a}s, Avenida Esperan\c{c}a s/n, Campus Samambaia, 74690-900, Goi\^{a}nia, Goi\'{a}s, Brazil}

\address[mymainaddress_sp]{São Paulo State University (UNESP), IBILCE, Campus S. J. Rio Preto, São Paulo, Brazil}

\begin{abstract}
We analyze a three-dimensional discontinuous piecewise linear system \(Z=(X,Y)\) whose switching manifold \(\Sigma\) contains visible-visible two-fold intersection lines. Assuming that the matrices \(DX\) and \(DY\) each have one nonzero real eigenvalue and one pair of complex conjugate eigenvalues, we reduce the system to a canonical form. Under a resonant condition, we use Darboux integrability theory to obtain a first integral common to \(X\) and \(Y\). Its restriction to \(\Sigma\) defines a hyperbola \(\Gamma\), which parametrizes the crossing points of symmetric periodic orbits. On this curve we construct the half-return maps, derive analytic expansions for the corresponding return times near infinity, and introduce a time-matching function given by their difference. By means of the Weierstrass Preparation Theorem, we prove the existence of a large-amplitude symmetric limit cycle for a suitable subfamily of systems. We then study stability through a saltation-corrected monodromy matrix and reduce the problem to Schur--Cohn inequalities for the two transverse Floquet multipliers.
\end{abstract}




\end{frontmatter}



\section{Introductory Remarks and Main Results}\noindent

Determining the existence, multiplicity, and spatial configuration of limit cycles constitutes one of the central pillars of qualitative dynamical systems theory. Initially formulated for smooth polynomial vector fields as part of Hilbert's 16th problem, this fundamental question remains a profound challenge and has acquired renewed complexity and significance for piecewise linear dynamical systems (PWLS). In such systems, although the flow in each region of the phase space is governed by a linear vector field, the global dynamics may exhibit non-smooth behavior due to transitions across switching manifolds. This allows for the emergence of qualitative dynamical phenomena---such as periodic orbits (see \cite{Freire, Freire2020_1, Freire2020_2, ferreira2026}), invariant surfaces (see \cite{Carmona, Freitas1, freitas3}), and global bifurcation phenomena (see \cite{cristiano2024, hosham2024, wang2024})---which are absent in smooth linear systems but become possible through the coupling of linear vector fields across switching manifolds.

The switching manifold constitutes a fundamental geometric object of the system. Its presence defines a characteristic structural mechanism: an abrupt change in the governing linear dynamics when trajectories cross the boundary separating distinct dynamical regimes in the phase space. This mechanism captures essential features in a variety of applied contexts, from idealized models for impacts and friction in engineering \cite{Brogliato, belykh} to threshold-driven activation in neuroscience \cite{goldbeter, vanSoest, coombes2018, coombes2023, rosch}. Consequently, the mathematical interest of PWLS lies in understanding how the interaction between simple linear subsystems and the global geometry of switching manifolds gives rise to qualitative dynamical phenomena, including the limit cycles that are the focus of this work. 

In this work we study a three-dimensional discontinuous piecewise linear vector field \(Z=(X,Y)\) with switching manifold
\[
\Sigma=\{(x,y,z)\in\mathbb{R}^{3}: z=0\}.
\]
We assume that \(\Sigma\) contains concurrent two-fold tangency lines and that the linear parts of \(X\) and \(Y\) each possess one real eigenvalue and one pair of complex conjugate eigenvalues with nonzero imaginary part. This is a non-generic but highly structured configuration, and precisely this structure makes it possible to obtain a canonical form adapted to both the geometry of the switching manifold and the symmetry properties imposed throughout the paper. After suitable changes of coordinates and parameters, the system is reduced to a normal form that allows a systematic analysis of periodic dynamics.

A central difficulty in the study of periodic solutions of discontinuous piecewise linear systems lies in the determination of the return mechanism across the switching manifold. If a periodic orbit crosses \(\Sigma\) transversely at two distinct points, then a local Poincar\'e return map can be defined near this orbit. In principle, this map is obtained by composing two transition maps: one associated with the flow of \(X\) in the region \(z\ge 0\), and the other associated with the flow of \(Y\) in the region \(z<0\). The construction of these maps requires the computation of the flight times needed for trajectories starting on \(\Sigma\) to return to \(\Sigma\). The periodicity condition is then expressed by requiring the two half-trajectories to reconnect after crossing the discontinuity. Although this procedure is conceptually natural, it is often difficult to implement explicitly, since the required return times usually cannot be written in closed form.

The geometric structure of the problem is clarified by two distinguished lines, \(r^{X}\) and \(r^{Y}\), obtained as the intersections of the invariant planes associated with the complex eigendirections of \(DX\) and \(DY\) with the switching manifold \(\Sigma\). We refer to these invariant planes as focal planes. The lines \(r^{X}\) and \(r^{Y}\) encode the tangency configuration at the origin and play a decisive role in the return mechanism. Their concurrency reflects the two-fold singular structure assumed in the model and distinguishes the class studied here from more generic configurations.

The periodic solutions of interest are those compatible with the involutive symmetry
\[
S(x,y,z)=(-y,-x,-z).
\]
Under this symmetry, the two intersection points of a symmetric periodic orbit with the switching manifold \(\Sigma\) are exchanged by the restriction of \(S\) to \(\Sigma\). Thus, the symmetry characterizes the periodic orbits under consideration and determines how their crossings with the switching manifold are related.

A second structural ingredient is the resonant restriction on the eigenvalues, namely that the modulus of the real eigenvalue equals twice the real part of the complex conjugate pair. This condition guarantees the existence of a common first integral for the two smooth subsystems. As a consequence, the periodic orbit problem can be reduced to the study of a conic
\[
\Gamma\subset\Sigma,
\]
which contains the crossing points of the periodic solutions under consideration. In this way, the search for periodic orbits is reduced from the switching plane to a one-dimensional geometric set. The common first integral is used here as a geometric reduction tool: it selects the conic \(\Gamma\) containing the crossing points of the relevant periodic orbits. 

Once this reduction is achieved, the closing condition can be formulated on \(\Gamma\) in terms of the half-return times of the vector fields \(X\) and \(Y\). More precisely, we analyze these times for initial conditions on \(\Gamma\), especially in a neighborhood of infinity along its hyperbolic branches. Using exponential matrix techniques, we show that the corresponding time-of-flight functions admit analytic asymptotic expansions. Their difference defines a time-matching function whose zeros correspond to symmetric periodic solutions. The Weierstrass Preparation Theorem then yields the existence of an isolated positive zero of this function, and consequently the existence of a large-amplitude limit cycle.

Once existence has been established, we study the stability of the limit cycle. In the piecewise-smooth setting, the derivative of the Poincar\'e map is not determined solely by the variational equations of the smooth subsystems, since one must also incorporate the jump in the linearized flow produced when trajectories cross the switching manifold; see \cite{dieci2011}. For this reason, we compute a saltation-corrected monodromy matrix and apply the Schur--Cohn criterion; see, for instance, p.~247 of \cite{Elaydi2005}. This allows us to identify a region of the parameter plane bounded by two curves, whose coordinates are given by the real part of the complex eigenvalues and the parameter associated with the intersections of the focal planes with the switching manifold, in which the limit cycle is attracting.

After establishing existence and attraction, we turn to the algorithmic aspect of the problem. Although the proof of existence is formulated through the reduced dynamics on the hyperbola \(\Gamma\), this same reduced setting also yields an effective numerical procedure for locating the attracting cycle. More precisely, within the stability band, by choosing an initial condition on \(\Gamma_1\) and iterating the reduced return map, one is led to the fixed point corresponding to the attracting limit cycle. In this way, the procedure provides a practical method for detecting the cycle and for generating a catalogue of limit cycles within the corresponding stability band.

The analysis developed in this paper is therefore organized around the interplay of symmetry and resonance. The symmetry characterizes the periodic orbits under consideration, while the resonant condition guarantees the existence of a common first integral. This, in turn, reduces the problem to the study of the conic \(\Gamma\), on which the time-matching condition can be formulated. From this reduced setting, we obtain the existence of limit cycles, derive stability conditions through saltation corrections and Floquet theory, and formulate an algorithmic procedure that produces a catalogue of cycles within the corresponding stability band.

We denote by \(\CL\) the class of piecewise linear vector fields \(Z=(X,Y)\) such that \(DX\) and \(DY\) each admit one real eigenvalue and one pair of complex conjugate eigenvalues with nonzero imaginary part. We denote by \(\CLperp\subset\CL\) the subclass for which the tangency lines \(r^{X}\) and \(r^{Y}\) are concurrent at the origin. Finally, \(\CLequi\subset\CLperp\) denotes the equivariant subclass with respect to the involution \(S\).

\begin{theorem}\label{teo:main1}
	Consider \(Z\in\CLequi\subset\CLperp\). Then the following statements hold:
	\begin{enumerate}
	\item[(a)] An orbit \(\gamma\) is a crossing symmetric periodic orbit of \(Z\) if and only if its successive intersection points \((x_0,y_0,0)\) and \((x_1,y_1,0)\) with the switching manifold \(\Sigma\) satisfy
	\[
		x_1=-y_0
		\qquad\text{and}\qquad
		y_1=-x_0.
	\]
	\item[(b)] If, in addition, the real eigenvalue is minus twice the real part of the associated complex pair, then there exist vector fields \(Z\in\CLequi\) admitting at least one symmetric limit cycle.
\item[(c)] Furthermore, under the hypotheses of item~(b), there exists an open subset of the \((C,H)\)-plane, called the stability region, in which the corresponding symmetric limit cycle is asymptotically stable. In particular, this stability persists under sufficiently small perturbations of the parameters within that region.
	\end{enumerate}
\end{theorem}

The paper is organized as follows. In Section \ref{sec:dvf} we introduce the basic definitions and auxiliary results. Section \ref{sec:cf} derives the canonical form for the class under consideration. In Section \ref{sec:caralimitcyle} we characterize symmetric periodic orbits and prove part~(a) of Theorem \ref{teo:main1}. Section \ref{sec:invcurve} constructs the algebraic curve \(\Gamma\) from the common first integral. In Section \ref{subsec:hra} we prove the analyticity of the half-return times along the relevant branches of \(\Gamma\). Section \ref{sec:time_matching} introduces the time-matching function and proves part~(b) of Theorem \ref{teo:main1}. The final section, Section \ref{subsec:stability_saltation}, develops the orbital stability theory and proves part~(c) of Theorem \ref{teo:main1}.

\section{Preliminary Results}\label{sec:dvf}
In this section, we define the objects and state the fundamental results used to prove our main theorems. For basic definitions on piecewise vector fields, we refer the reader to \cite{marco}.
\begin{definition}
	Let $\mathfrak{X}$ be the set of all polynomial vector fields of degree one on $\mathbb{R}^3$. Consider  $X=\left(X_1,X_2,X_3\right) \in \mathfrak{X}$, with $X_j(x,y,z)=A_jx+B_jy+C_jz+D_j$ where $A_j,B_j,C_j,D_j$ are real constants for $j=1,2,3.$
\end{definition}
\begin{definition}
	A point $p\in\mathbb{R}^3$ is a singularity of $X\in\mathfrak{X}$ if $X\left(p\right)=0$. 
\end{definition}
In the following definition, we establish the class of vector fields we consider.

\begin{definition}\label{def:cl2z}
	Let $\CL$ be the set of the piecewise linear vector fields $Z$, 
	\begin{equation}\label{eq:geral}
		Z(\mathbf{x})=\left(X(\mathbf{x}),Y(\mathbf{x})\right)=\begin{cases}
			X(\mathbf{x})\ \text{if} \ \mathbf{x} \in \Sigma^{+},
			\\
			Y(\mathbf{x})\ \text{if} \ \mathbf{x} \in \Sigma^{-},
		\end{cases}
	\end{equation}
	where $\mathbf{x}=\left(x,y,z\right)$, $X,Y\in\mathfrak{X}$ such that the matrices $DX$ and $DY$ each have a nonzero real eigenvalue and a pair of conjugate complex eigenvalues with a nonzero imaginary part. The  switching set is the plane $\Sigma=\{\mathbf{x}\in\mathbb{R}^3 \ | \ f\left(\mathbf{x}\right)=z=0\}$,  $\Sigma^{+}=\{\mathbf{x}\in\mathbb{R}^3 \ | \ f\left(\mathbf{x}\right)\geq 0\}$ and  $\Sigma^{-}=\{\mathbf{x}\in\mathbb{R}^3 \ | \ f\left(\mathbf{x}\right)< 0\}$.
	
\end{definition}

The switching plane $\Sigma$ contains three distinct regions, determined by the way the vector fields $X$ and $Y$ meet $\Sigma$. Accordingly, the classification of points $p\in\Sigma$ depends on the Lie derivatives $Xf(p)=\left<X(p),\nabla f(p)\right>$ and $Yf(p)=\left<Y(p),\nabla f(p)\right>$. Therefore, we have:

\begin{enumerate}
	\item The \textit{crossing region} $\Sigma^{\rm c} = \{p\in\Sigma \ | \ Xf(p)\,Yf(p)>0\}$;
	\item The \textit{sliding region} $\Sigma^{\rm s}=\{p\in\Sigma \ | \ Xf(p)<0 \quad \text{and} \quad Yf(p)>0\}$;
	\item The \textit{escaping region} $\Sigma^{\rm e}=\{p\in\Sigma \ | \ Xf(p)>0 \quad \text{and} \quad Yf(p)<0\}$.
\end{enumerate}
Here, “c”, “s”, and “e” are abbreviations for “crossing”, “sliding”, and “escaping”, respectively.
The solutions in the sliding and escaping regions are constructed using Filippov's convex method \cite{Filippov}. This method involves considering the convex combination $Z^{\rm s}\left(p\right)$ between the vector fields $X$ and $Y$ at each point $p\in\Sigma^{\rm s}\cup\Sigma^{\rm e}$, given by:
\begin{equation}
Z^{\rm s}(p)=\dfrac{Yf(p)X(p)-Xf(p)Y(p)}{Yf(p)-Xf(p)}.
\end{equation}
The vector field $Z^{\rm s}$ is called the \textit{sliding vector field} and the points $x\in\Sigma^{\rm s}\cup\Sigma^{\rm e}$ such that $Z^{\rm s}(x)=0$ are called a \textit{pseudo-equilibrium of} \eqref{eq:geral}.

The \textit{tangency points} are the points of $\Sigma$ where either of the vector fields $X$ or $Y$ is tangent to $\Sigma$. These points are found on the boundaries $\partial\Sigma^{\rm c}\cup\partial\Sigma^{\rm s}\cup\partial\Sigma^{\rm e}$ and satisfy $Xf\left(p\right)=0$ or $Yf\left(p\right)=0$. Since the vector fields are polynomials of degree one, these tangency points form lines, referred to as \textit{tangency lines}. We denote the tangency lines of vector fields $X$ and $Y$ by $L_X$ and  $L_Y$, respectively.
\begin{definition}\label{def:tangencia_quadratica}
	A point $p\in L_{X}$ ($q\in L_{Y}$) is said to be a fold if $X^2f(p)\neq0$ ($Y^2f(q)\neq0$). A fold $p\in L_{X}$  ($q\in L_{Y}$) is called a visible fold point for the vector field $X$ (resp., $Y$) if $X^2f(p)>0$ (resp., $Y^2f(q)<0$) and invisible fold if $X^2f(p)<0$ (resp., $Y^2f(q)>0$). A point $p\in\Sigma$ is a double fold  point if $p$ is a tangency point for both vector fields $X$ and $Y$. A point $p\in L_{X}$ ($q\in L_{Y}$) is said to be a cusp if $X^2f(p)=0$ and $X^3f(p)\neq0$ ($Y^2f(q)=0$ and $Y^3f(q)\neq0$).
\end{definition}

We consider $Z=(X,Y)\in \CL$, where \(X(\mathbf{x}) = \sum_{j=1}^{3} (A_j x + B_j y + C_j z + D_j) \frac{\partial}{\partial x_j}\) and \(Y(\mathbf{x}) = \sum_{j=1}^{3} (a_j x + b_j y + c_j z + d_j) \frac{\partial}{\partial x_j}\), with \(\mathbf{x} = (x, y, z)\). The tangency lines are then given by \(L_X = \{p(x, y, z) \in \Sigma: A_3 x + B_3 y + D_3 = 0\}\) and \(L_Y = \{p(x, y, z) \in \Sigma: a_3 x + b_3 y + d_3 = 0\}\). Note that within \(\Sigma\), the lines \(L_X\) and \(L_Y\) can be concurrent, parallel, or coincident.

The primary objective is to study the existence of limit cycles. Therefore, the coefficients of $X_3$ and $Y_3$ must satisfy $A_3^2+B_3^2\neq 0$ and  $a_3^2+b_3^2\neq 0$, implying that the subsets $L_X$ and $L_Y$ must not be empty. In other cases, all points of $\Sigma$ are either tangential or not, i.e., there are no limit cycles.

Consider $X\in\mathfrak{X}$ and the differential system 
\begin{equation}\label{eq:sistema_diferencial}
	\dot{\mathbf{x}}\left(t\right)=X\left(\mathbf{x}\left(t\right)\right),
\end{equation}
where $\mathbf{x}=\left(x,y,z\right)$.
Let $\varphi_X\left(t,p_0\right)$ be the flow of \eqref{eq:sistema_diferencial} for a point $p_0\in\mathbb{R}^3$. Thus, we have
$$
\begin{cases}
	\dfrac{d}{dt}\varphi_X\left(t,p_0\right)=X\left(\varphi_X\left(t,p_0\right)\right),\\
	\varphi_X\left(0,p_0\right)=p_0.
\end{cases}
$$
The flow $\varphi_X(t,p_0)$ is defined for $t\in I\subset\mathbb{R}$, where $I=I(p_0,X)$  is a real interval which depends on the point $p_0$ and the vector field $X$.
\begin{definition}
	The trajectory of a vector field $Z=\left(X,Y\right)\in\CL$ is given by:
	\begin{itemize}
		\item if $p_0\in\Sigma^{+}$ and $X(p_0)\neq0$, then $\varphi_Z\left(t,p_0\right)=\varphi_X\left(t,p_0\right)$, for $t\in I(p_0,X)\subset\mathbb{R}$;
		\item if $p_0\in\Sigma^{-}$ and $Y(p_0)\neq0$, then $\varphi_Z\left(t,p_0\right)=\varphi_Y\left(t,p_0\right)$, for $t\in I(p_0,Y)\subset\mathbb{R}$;
		\item  if $p_0\in\Sigma^{\rm c}$ and $Xf(p_0)>0$ and $Yf(p_0)>0$, then 
		$$
		\varphi_Z\left(t,p_0\right)=\begin{cases}
			\varphi_X\left(t,p_0\right)\quad t\geq0,\\
			\varphi_Y\left(t,p_0\right)\quad t<0.
		\end{cases}
		$$
		If $Xf(p_0)<0$ and $Yf(p_0)<0$, consider the reverse time;
		\item If \(p_0 \in \partial\Sigma^{\rm c} \cup \partial\Sigma^{\rm s} \cup \partial\Sigma^{\rm e}\) and the definitions of the trajectories for points in \(\Sigma\) on both sides of \(p_0\) can be extended to \(p_0\) and coincide, the trajectory passing through \(p_0\) is this trajectory. These points are called regular tangencies.
		\item  $\varphi_Z\left(t,p_0\right)=p_0$, $t\in\mathbb{R}$ for all other points $p_0\in\mathbb{R}^3$. This case includes the singular tangencies points in $\Sigma$ (not regular), the singular points of the vector fields $X$ or $Y$ or $Z^{\rm s}$, and the invisible-invisible folds.
	\end{itemize}
\end{definition}
\begin{definition}
	The local orbit $\gamma(p_0)$ of a point $p_0\in\mathbb{R}^3$ for the vector field $Z$ is defined as the set of all points reachable from $p_0$ by following the flow $\varphi_Z$, that is, $\gamma(p_0)=\{\varphi_Z\left(t,p_0\right)\ | \ t\in I\}$.
\end{definition}
Next, we define the concept of the fundamental matrix, following \cite{dieci2011}, that we use to determine the closing equations.

To analyze local orbital stability, we consider the Poincar\'e first-return map associated with the crossing periodic orbit. Its linearization is described by the variational flow along each smooth segment of the orbit, combined with the jump corrections at the switching events through the corresponding saltation matrices; see \cite{dieci2011}.

\begin{definition}\label{def:matriz_fundamental}
	Let \(\varphi_X(t,p_0)\) be the flow of a \(C^1\) vector field \(X\). The fundamental matrix along the trajectory through \(p_0\) is defined by
	\[
	\Phi_X(t,0):=D_{p_0}\varphi_X(t,p_0),
	\]
	and satisfies the variational equation
	\[
	\begin{cases}
		\Phi_X'(t,0)=DX(\varphi_X(t,p_0))\,\Phi_X(t,0),\\[2mm]
		\Phi_X(0,0)=I.
	\end{cases}
	\]
\end{definition}

The fundamental matrix is nonsingular and satisfies the usual composition property. In particular, it maps the vector field along the trajectory according to
\begin{equation}\label{eq:mapeamento_matriz_fundamental}
	X(\varphi_X(t,p_0))=\Phi_X(t,0)\,X(p_0),
\end{equation}
which is the standard flow-invariance property of the variational equation.

Let \(\Pi\) denote the local first-return map near a crossing periodic orbit, and let \(p_0\in\Sigma\) be one of its intersection points with the switching manifold. Then
\[
M=D\Pi(p_0)
\]
is the corresponding monodromy matrix. For a crossing periodic orbit, \(M\) is obtained by composing the smooth variational flows with the saltation matrices at the switching instants, namely
\begin{equation}\label{eq:monodromy_saltation}
	M=S_{Y\to X}(p_0)\,\Phi_Y(t_Y)\,S_{X\to Y}(p_1)\,\Phi_X(t_X),
\end{equation}
where \(p_1\) is the next intersection point with \(\Sigma\), and
\[
\Phi_X(t_X)=e^{DX\,t_X},
\qquad
\Phi_Y(t_Y)=e^{DY\,t_Y}.
\]

Since \(\Sigma=\{f=0\}\) with \(f(x,y,z)=z\), the saltation matrices are given by
\[
S_{X\to Y}
=
I+\frac{(Y-X)}{Xf}\,e_3^T,
\qquad
S_{Y\to X}
=
I-\frac{(Y-X)}{Yf}\,e_3^T,
\]
where
\[
Xf=\nabla f^T X,
\qquad
Yf=\nabla f^T Y.
\]

Moreover, by \eqref{eq:mapeamento_matriz_fundamental} and its piecewise-smooth counterpart, the generalized monodromy matrix preserves the tangent direction of the periodic orbit. Hence, if \(\gamma(t)\) is a periodic orbit of period \(T\), then
\[
M\,Z(p_0)=Z(p_0),
\]
so that \(\mu_1=1\) is the trivial Floquet multiplier associated with time-translation invariance. Therefore, orbital stability is determined by the two nontrivial multipliers.

The mapping property \eqref{eq:mapeamento_matriz_fundamental} of the fundamental matrix plays a crucial role in our analysis, as it will be used to define our closing equations. 

Since $Z\in\CL$, it follows that the exponential matrices ${\rm e}^{DX\,t}$ and ${\rm e}^{DY\,(-t)}$ are fundamental matrices for the vector fields $X$ and $Y$, respectively. 

\begin{definition}\label{def:eq_fechamento}
Consider \( Z = (X, Y) \in \CL \). Given that the matrices \( DX \) and \( DY \) each possess a non-zero real eigenvalue, along with complex eigenvalues having non-zero imaginary parts, we define the closing equations for the vector fields \( X \) and \( Y \) as  
\[
{\rm e}^{DX\,t^{X}}X(x_0, y_0, 0) = X(x_1, y_1, 0) \quad \text{and} \quad {\rm e}^{DY\,(-t^{Y})}Y(x_0, y_0, 0) = Y(x_2, y_2, 0),
\]
where \( (x_0, y_0, 0) \) is the initial point and \( y_0 > 0 \) is chosen sufficiently large to ensure the existence of return times \( t^{X} > 0 \) and \( t^{Y} > 0 \), as well as return points \( (x_1, y_1, 0) \) and \( (x_2, y_2, 0) \) for the vector fields \( X \) and \( Y \), respectively.
\end{definition}

In order to find explicit first integrals of the vector fields involved, we introduce the notion of Darboux polynomials. These play a central role in identifying invariant algebraic surfaces and constructing first integrals, especially in systems with polynomial structure.

\begin{definition}\label{def:darbouxpolynomial}
A polynomial \( f(x, y, z) \in \mathbb{C}[x, y, z] \) is called a \emph{Darboux polynomial} for the vector field \( X \) if there exists a polynomial \( k(x, y, z) \in \mathbb{C}[x, y, z] \) such that \( X(f) = kf \). In this case, \( k \) is called the \emph{cofactor} of \( f \). The algebraic surface \( \{ f = 0 \} \) is then invariant under the flow of \( X \).
\end{definition}

It is worth noting that for a polynomial vector field of degree \( d \), the degree of any cofactor \( k \) will not exceed \( d - 1 \).



The following proposition provides a classical criterion, known as the Darboux integrability condition, for constructing first integrals from a finite number of Darboux polynomials.

\begin{proposition}\label{prop:firstintegral}
     Assume that a vector field \(X\) admits \(n\) Darboux polynomials \(f_i\) with cofactors \(k_i\) for \(i = 1, \ldots, n\). Then, a function \(F = f_1^{\lambda_1} \cdots f_n^{\lambda_n}\) is a first integral of \(X\) if and only if there exist constants \(\lambda_1, \ldots, \lambda_n\), not all zero, such that \(\sum_{i=1}^{n} \lambda_i k_i = 0\). Such a function \(F\) is called a Darboux function.
\end{proposition}
\begin{proof}
(\textbf{$\Rightarrow$}) Suppose that \( \sum_{i=1}^n \lambda_i k_i = 0 \). Define the function
\[
F := f_1^{\lambda_1} \cdots f_n^{\lambda_n}.
\]
We compute the Lie derivative of \( F \) along the vector field \( X \). Using the chain rule for logarithmic derivatives, we have:
\[
X(F) = F \cdot \sum_{i=1}^n \lambda_i \frac{X(f_i)}{f_i}.
\]
Since \( f_i \) is a Darboux polynomial with cofactor \( k_i \), we have \( X(f_i) = k_i f_i \), and therefore
\[
\frac{X(f_i)}{f_i} = k_i.
\]
Substituting, we get
\[
X(F) = F \cdot \sum_{i=1}^n \lambda_i k_i = F \cdot 0 = 0.
\]
Thus, \( F \) is constant along the trajectories of \( X \), i.e., \( F \) is a first integral.

\bigskip
\textbf{($\Leftarrow$)} Conversely, suppose that \( F = f_1^{\lambda_1} \cdots f_n^{\lambda_n} \) is a first integral, so that \( X(F) = 0 \). As above,
\[
X(F) = F \cdot \sum_{i=1}^n \lambda_i k_i.
\]
Since \( F \neq 0 \), it must be that
\[
\sum_{i=1}^n \lambda_i k_i = 0.
\]

Therefore, the function \( f_1^{\lambda_1} \cdots f_n^{\lambda_n} \) is a first integral of \( X \) if and only if \( \sum_{i=1}^n \lambda_i k_i = 0 \).
\end{proof}

The next results will be used to study analytic function zeros. We first state a version of the Implicit Function Theorem for analytic functions to guarantee the existence of an analytic function implicitly defined by an equation  $F(x,y)~=~0$, see \cite{Krantz[2013]}.

\begin{theorem}[Implicit Function Theorem]\label{teo:funcao_implicita}
	Suppose that the  power series
	$$
	F\left(x,y\right)=\displaystyle\sum_{k,n}a_{k,n}x^ky^n,
	$$ is absolutely convergent to $|x|\leq R_1$ and $|y|\leq R_2$. If $a_{0,0}=0$ and $a_{0,1}\neq0$,
	then there is $r_0>0$ such that the power series 
	$$
	f(x)=\displaystyle\sum_{n>0}c_nx^n,
	$$
	is absolutely convergent to $|x|\leq r_0$ and
	$$
	F(x,f(x))=0.
	$$
\end{theorem}

Let $f(z, p)$ be an analytic function of multiple variables, where $z \in \mathbb{C}$ and $p = (p_1 ,\ldots,p_n) \in\mathbb{C}^n$. If $f\left(z_0 , p_0\right)=0$, the Weierstrass Preparation Theorem (see \cite{WPT_2001})  provides a representation of the function $f(z, p)$ in a neighborhood of the point $\left(z_0, p_0\right)$ as a product of two factors: one is a polynomial in $z-z_0$ and the other factor is nonzero at $\left(z,p\right)=\left(z_0,p_0\right)$. Therefore, Weierstrass Preparation Theorem ``prepares" the function $f$ for a study of its zeros, because the zeros of the function $f(z,p)$ near $\left(z,p\right)=\left(z_0,p_0\right)$ are precisely the zeros of the polynomial factor of the representation obtained from the conclusion of the theorem.

\begin{theorem}[Weierstrass Preparation Theorem]\label{teo:teorema_preparacao_weierstrass}
	Suppose that $G:\mathbb{C}\times\mathbb{C}^n\to\mathbb{C}$ is an analytic function such that there exists $\left(x_0,p_0\right)\in\mathbb{C}\times\mathbb{C}^n$ so that
	$$
	G\left(x_0,p_0\right)=\dfrac{\partial G}{\partial x}\left(x_0,p_0\right)=\cdots=\dfrac{\partial^{m-1}G}{\partial x^{m-1}}\left(x_0,p_0\right)=0, \qquad \dfrac{\partial^{m}G}{\partial x^{m}}\left(x_0,p_0\right)\neq0.
	$$ Therefore, there is a neighborhood $U_0\in \mathbb{C}\times\mathbb{C}^n$ of the point $\left(x_0,p_0\right)$, so that   
	\begin{equation}\label{eq:teorema_preparacao_weierstrass}
		G(x,p)=H(x,p)\left(\left(x-x_0\right)^m+\sum_{n=0}^{m-1}A_n(p)\left(x-x_0\right)^n\right),
	\end{equation}
	where $A_n(p)$ and $H(x,p)$ are analytic functions defined uniquely by the function $G$ and $A_n(p_0)=0$ and $H(x,p)\neq0 $ for all $\left(x,p\right)\in U_0$.
\end{theorem}

\section{Canonical Form}\label{sec:cf}
In this section, we obtain the canonical form for our class of 3D piecewise linear vector fields.\begin{definition}
	Let $\CLperp\subset\CL$ be the set of vector fields $Z=(X,Y)$ such that the tangency lines $L_{X}$ and $L_{Y}$ are concurrent at the origin of $\mathbb{R}^3$.
\end{definition}
The following proposition establishes a suitable canonical form for systems in this class.

\begin{proposition}
	Consider $Z\in \CLperp$. There exists a linear change of variables that transforms $Z$ into
		\begin{equation}\label{eq:canonica_ortogonal}
		Z(\mathbf{x})	=	\begin{cases}
				X(\mathbf{x})=(A\,x -H\left(\left(\left(A -C \right)^{2}+1\right)z-\Lambda\right),\Lambda-(1+C^{2})z,2C\,z+y)  &\text{if} \ \mathbf{x}\in\Sigma^{+},
				\\
				Y(\mathbf{x})=(\lambda-(1+c^{2})z,ay -h\left(\left(\left(a-c\right)^2+1\right)z-\lambda\right),2c\,z+x)  &\text{if} \ \mathbf{x}\in\Sigma^{-},
		\end{cases}
	\end{equation}
	where $\mathbf{x}=\left(x,y,z\right)$.
\end{proposition}

\begin{proof}
	Since we assume that the tangency lines $L_{X}$ and $L_{Y}$ are concurrent in $\Sigma$, it follows that the linear system
	$$
	\begin{cases}
		A_3x+B_3y+D_3=0,\\
		a_3x+b_3y+d_3=0,
	\end{cases}
	$$
	has a unique solution. Consequently, $m=A_{3} b_{3}-B_{3} a_{3}\neq0$.
	Without loss of generality, we may assume that this intersection occurs at the origin.
	Thus, $D_{3}=d_{3}=0$.

	By performing the change of variables $u=\dfrac{A_3y-B_3x}{m}$ and $v=\dfrac{b_3x-a_3y}{m}$ and then renaming the variables, the system \(Z\) can be written as
	$$
	Z\left(\mathbf{x}\right)=	\begin{cases}
		X(\mathbf{x})=(A_1x+B_1y+C_1z+D_1,B_2y+C_2z+D_2,C_3z+y)  &\text{if} \ z\geq0,
		\\
		Y(\mathbf{x})=(a_1x+c_1z+d_1,a_2x+b_2y+c_2z+d_2,c_3z+x)   &\text{if} \ z\leq0.
	\end{cases}
	$$
	It follows that
	$$
	L_{X}=\left\{(x,0,0)\ \Big|\ x\in\mathbb{R}\right\} \qquad \text{and} \qquad L_{Y}=\left\{(0,y,0)\ \Big|\ y\in\mathbb{R}\right\}.
	$$

	Now, computing the second Lie derivatives of $X$ and $Y$, we obtain
	$
	X^2f(p)=A_2x+ D_{2}\ \text{for} \ p\in L_{X}
	$
	and
	$ Y^2f(q)=b_1y+ d_{2}\ \text{for}\ q\in L_{Y}.
	$
	Under the assumption that all tangencies are of fold type, it follows that
	$A_2=0$ and $b_1=0$.

	Consequently,
	$$
	Z\left(\mathbf{x}\right)=	\begin{cases}
		X(\mathbf{x})=(A_1x+B_1y+C_1z+D_1,B_2y+C_2z+D_2,C_3z+y)  &\text{if} \ z\geq0,
		\\
		Y(\mathbf{x})=(a_1x+c_1z+d_1,a_2x+b_2y+c_2z+d_2,c_3z+x)   &\text{if} \ z\leq0.
	\end{cases}
	$$

	Now, setting $u=x+B_1z$ and $v=y+a_2z$ and renaming the variables, we obtain
	\begin{equation}\label{eq:quase_canonica}
		Z\left(\mathbf{x}\right)=	\begin{cases}
			X(\mathbf{x})=(A_1x+C_1z+D_1,B_2y+C_2z+D_2,C_3z+y)  &\text{if} \ z\geq0,
			\\
			Y(\mathbf{x})=(a_1x+c_1z+d_1,b_2y+c_2z+d_2,c_3z+x)   &\text{if} \ z\leq0.
		\end{cases}
	\end{equation}

As we assume that $DX$ and $DY$ admit complex eigenvalues with non-zero imaginary parts, the dynamics in each region exhibit rotational behavior around the eigenspace generated by the real eigenvalue. We scale the frequencies of the complex eigenvalues to $1$ for both dynamics, so that the expansion or contraction of each focus depends only on the real part of the corresponding complex eigenvalue. Therefore, system~\eqref{eq:quase_canonica} can be rewritten in terms of the eigenvalues of the matrices $DX$ and $DY$, namely, $A$ and $a$, which correspond to the real eigenvalues, and $C\pm i$ and $c\pm i$, which denote the pairs of complex conjugate eigenvalues, respectively, by setting
	\[
A_1=A,\qquad	B_{2}=0,\qquad C_{2}=-C^{2}-1,\qquad C_{3}=2C
	\]
	and 
	\[
	a_1=0,\qquad b_2=a, \qquad c_1=-c^2-1,\qquad c_3=2c.
	\]
	
As previously discussed, the angular and linear coefficients of the lines $r^X$ and $r^Y$ play a significant role in our analysis. To simplify their equations, we set
	\[
	C_1 = -\left((A - C)^2 + 1\right)H,
	\qquad
	D_1 = \Lambda\,H,
	\] and
	\[
	c_2=-\left((a - c)^2 + 1\right)h
	\qquad
	d_2=\lambda\,h.
	\]
Finally, renaming \[D_2=\Lambda, \qquad d_2=\lambda\] we get the canonical form \eqref{eq:canonica_ortogonal}. 
\end{proof}
Given $Z \in \CLperp$, we can express $Z$ in the canonical form~\eqref{eq:canonica_ortogonal}. In this representation, the parameters $A$ and $C$ (respectively, $a$ and $c$) denote the real eigenvalue and the real part of the complex pair of eigenvalues of $DX$ (respectively, $DY$). The parameters $H$ and $h$ determine the inclination of the lines $r^{X}$ and $r^{Y}$ on the switching manifold $\Sigma$. The parameters $\Lambda$ and $\lambda$ represent the second Lie derivatives of the vector fields $X$ and $Y$, respectively, and are associated with the visibility of the corresponding fold tangencies.

This choice simplify the computation of the first integral. In this setup, the zero level sets of the Darboux polynomials for the vector fields \(X\) and \(Y\) correspond to the invariant planes \(W^X\) and \(W^Y\), which are generated by the real and imaginary parts of the eigenvector associated with the complex eigenvalue of each field. The intersection of the planes \(W^X\) and \(W^Y\) with the discontinuity surface \(\Sigma\) yields two lines, \(r^X\) and \(r^Y\), which will play an important role in the hypotheses of our results. 

\begin{definition}[Equivariance]\label{def:equigeral}
	Let $Z=(X,Y)$ be a piecewise linear vector field with switching manifold
	\[
	\Sigma=\{(x,y,z)\in\mathbb{R}^3 \mid z=0\},
	\]
	and let
	\[
	S:\mathbb{R}^3\to\mathbb{R}^3
	\]
	be a linear involution satisfying $S(\Sigma)=\Sigma$ and $S(\Sigma^+)=\Sigma^-$.  
	We say that $Z$ is \emph{$S$-equivariant} if
	\begin{equation}\label{eq:def_equiv_final}
	X\big(S(\mathbf{x})\big)=S\big(Y(\mathbf{x})\big), \qquad \forall\,\mathbf{x}\in\Sigma^+.
	\end{equation}
\end{definition}

\begin{definition}\label{def:equiva}
	Let $\CLequi\subset\CLperp$ be the set of vector fields $Z=(X,Y)$ that are equivariant under the involutive symmetry
	\[
	S(x,y,z)=(-y,-x,-z).
	\]
\end{definition}

\begin{proposition}\label{prop:equivariant}
	A vector field $Z=(X,Y)\in\CLequi$ if and only if
	\[
	a=A,\qquad c=C,\qquad h=H,\qquad \lambda=-\Lambda.
	\]
\end{proposition}

\begin{proof}
Let $Z=(X,Y)$ be given by the canonical form~\eqref{eq:canonica_ortogonal}.  
By Definition~\ref{def:equigeral}, since $S(\Sigma^+)=\Sigma^-$, equivariance of
the piecewise vector field $Z$ is characterized by the condition
\[
X\big(S(\mathbf{x})\big)=S\big(Y(\mathbf{x})\big),
\qquad \forall\,\mathbf{x}\in\Sigma^+.
\]

Let $\mathbf{x}=(x,y,z)\in\Sigma^+$. A direct substitution of the canonical
expressions of $X$ and $Y$ given in~\eqref{eq:canonica_ortogonal} yields
\[
X(-y,-x,-z)=\Big(-Ay-H\big((A-C)^2+1\big)z-H\Lambda,\,
\Lambda+(1+C^2)z,\,
-2Cz-x\Big),
\]
and
\[
S\big(Y(x,y,z)\big)=\Big(-ay+h\big((a-c)^2+1\big)z-h\lambda,\,
-\lambda+(1+c^2)z,\,
-2cz-x\Big).
\]
Requiring equality of these expressions for all $(x,y,z)$ implies
\[
a=A,\qquad c=C,\qquad h=H,
\]
and, by comparing the constant terms in the second component,
\[
\lambda=-\Lambda.
\]

Conversely, if the relations
\[
a=A,\qquad c=C,\qquad h=H,\qquad \lambda=-\Lambda
\]
hold, then a direct substitution into~\eqref{eq:canonica_ortogonal} shows that
\[
X\big(S(\mathbf{x})\big)=S\big(Y(\mathbf{x})\big), \qquad \forall\,\mathbf{x}\in\Sigma^+,
\]
and hence $Z$ is $S$-equivariant. This concludes the proof.
\end{proof}

\begin{remark}
Under the conditions given in Proposition~\ref{prop:equivariant}, the involution
$S$ maps the geometric structures associated with the vector field $X$ onto those
associated with $Y$, namely,
\[
S(W^X)=W^Y \quad \text{and} \quad S(r^X)=r^Y.
\]
In particular, the inclination of the tangency lines on the switching manifold
$\Sigma$ and the spectral structure of the linear dynamics are preserved by the
symmetry.
\end{remark}

As a consequence of Proposition~\ref{prop:equivariant}, in the equivariant class
$\CLequi$ the piecewise vector field $Z=(X,Y)$ is completely determined by a
single linear vector field. More precisely, once the restriction of $Z$ to
$\Sigma^+$ is fixed, the vector field on $\Sigma^-$ is uniquely determined by
the equivariance condition.

Indeed, for $Z\in\CLequi$, the vector field $Y$ is obtained from $X$ through
\begin{equation}\label{eq:Y_from_X}
	Y(\mathbf{x}) = S\big(X(S(\mathbf{x}))\big), \qquad \mathbf{x}\in\Sigma^-.
\end{equation}
Consequently, the study of equivariant systems in $\CLequi$ can be reduced to
the analysis of a single linear vector field defined on $\Sigma^+$.

\begin{remark}
	The fixed-point set of the involution $S(x,y,z)=(-y,-x,-z)$ is given by
	\[
	\Fix(S)=\{\mathbf{x}\in\mathbb{R}^3 \mid S(\mathbf{x})=\mathbf{x}\}
	=\{(x,-x,0)\mid x\in\mathbb{R}\}.
	\]
	In particular, $\Fix(S)\subset\Sigma$ and is invariant under the flow generated by any $S$-equivariant vector field. As a consequence, symmetric trajectories and periodic orbits, when they exist, may intersect the switching manifold along this fixed-point line.
\end{remark}

Assuming that the vector fields belong to the equivariant class \(\CLequi\) (Definition~\ref{def:equiva} and Proposition~\ref{prop:equivariant}), and let \([S]\) be the matrix associated with the involution
\[
S(x,y,z)=(-y,-x,-z),
\qquad
[S]=
\begin{pmatrix}
	0 & -1 & 0\\
	-1 & 0 & 0\\
	0 & 0 & -1
\end{pmatrix}.
\]
At the linear level, equivariance yields
\[
DX\,[S]=[S]\,DY,
\qquad
DY=[S]DX[S].
\]
Therefore \(DX\) and \(DY\) are similar, and
\begin{equation}\label{eq:fundamental_conjugacy}
	\Phi_Y(t)=e^{DYt}=[S]\,e^{DXt}\,[S]=[S]\Phi_X(t)[S].
\end{equation}
Equivalently, this follows from the flow identity
\[
\varphi_Y(t,p)=S\bigl(\varphi_X(t,S(p))\bigr),
\]
after differentiation with respect to the initial condition.

\section{Characterization of Limit Cycles}\label{sec:caralimitcyle}

\begin{proof}[Proof of part (a) of Theorem \ref{teo:main1}]
	Since $Z\in\CLequi$, the vector fields $X$ and $Y$ are related by the involution
	$S$ through
	\[
	Y(\mathbf{x})=S\big(X(S(\mathbf{x}))\big).
	\]
	Therefore, the flow of $Y$ is the image under $S$ of the flow of $X$.
	
	Let $(x_0,y_0,0)\in\Sigma$ and assume that the trajectory starting at this point
	evolves under the flow of $X$ until it reaches $\Sigma$ again at
	$(x_1,y_1,0)$. The subsequent evolution under the flow of $Y$ is then uniquely
	determined as the $S$-image of the first trajectory segment.
	
	For the full trajectory to close and form a periodic orbit, the second segment
	must return exactly to the initial condition $(x_0,y_0,0)$. By equivariance,
	this occurs if and only if
	\[
	(x_1,y_1,0)=S(x_0,y_0,0).
	\]
	Since $S(x_0,y_0,0)=(-y_0,-x_0,0)$, the stated conditions follow.
	
	Conversely, if these relations hold, the two trajectory segments are exchanged
	by $S$ and the orbit closes after two symmetric passages through $\Sigma$,
	yielding a periodic orbit. This proves that all periodic orbits in $\CLequi$
	are necessarily symmetric.
\end{proof}

\begin{remark}\label{rem:nonexistence_symmetry}
It is important to emphasize that the equivariance of the system does not, by itself, guarantee the existence of periodic orbits. However, part (a) of Theorem \ref{teo:main1} shows that, if periodic orbits exist, they must necessarily be symmetric with respect to the involution $S$. Moreover, such periodic orbits are also time-symmetric, since they are composed of two trajectory segments exchanged by the symmetry and traversed in opposite directions of the phase space. Hence, in the equivariant class $\CLequi$, symmetry is not a sufficient condition for the existence of periodic dynamics, but it fully determines the possible closure mechanism of trajectories.
\end{remark}

In view of Remark~\ref{rem:nonexistence_symmetry}, the characterization provided by
Theorem~\ref{teo:main1}(a) yields necessary geometric conditions for the closure of
trajectories, but not a sufficient existence criterion. We therefore proceed by
constructing suitable first integrals and restricting the return dynamics to the
corresponding invariant curve on \(\Sigma\). This reduction yields an analytic condition
which, combined with the symmetry characterization, provides a sufficient criterion for
the existence of a $T$-symmetric periodic orbit under the stated hypotheses.

\section{First Integrals and the Hyperbola}\label{sec:invcurve}
\subsection{Darboux polynomials}
\begin{proposition}\label{prop:firstintegral_X}\label{prop:firstintegral_Y}
Given $Z=(X,Y)\in\CLperp$, the following statements hold:
\begin{enumerate}
\item For the vector field $X$, the polynomials
$$f_1(x,y,z)=y^{2}+2 C \left(z +\dfrac{\Lambda}{C^{2}+1}\right) y +\left(C^{2}+1\right) z^{2}+\dfrac{2\Lambda\left(C^{2}-1\right)   z}{C^{2}+1}+\dfrac{\Lambda^{2}}{C^{2}+1}$$
and
$$f_2(x,y,z)=x-H\left(Az+y\right),$$
are Darboux polynomials with cofactors $2C$ and $A$, respectively. Consequently,
\begin{equation}\label{eq:firstintegral_X}
P_X(x,y,z)=f_1(x,y,z) f_2(x,y,z)^{-\frac{2C}{A}},
\end{equation}
is a first integral for $X$. Additionally,  \(r^X:=\left\{\left(x,y,z\right)\in\mathbb{R}^3\,|\, x=Hy\right\}\).

\item For the vector field $Y$, the polynomials
$$F_1(x,y,z)=x^{2}+2 c \left(z +\dfrac{\lambda}{c^{2}+1}\right) x +\left(c^{2}+1\right) z^{2}+\dfrac{2 \lambda  \left(c^{2}-1\right) z}{c^{2}+1}+\dfrac{\lambda^{2}}{c^{2}+1}$$
and
$$F_2(x,y,z)=y-h(az-x),$$
are Darboux polynomials with cofactors $2c$ and $a$, respectively. Consequently,
\begin{equation}\label{eq:firstintegral_Y}
P_Y(x,y,z)=F_1(x,y,z) F_2(x,y,z)^{-\frac{2c}{a}},
\end{equation}
is a first integral for $Y$. Additionally,  \(r^Y:=\left\{\left(x,y,z\right)\in\mathbb{R}^3\,|\, y=hx\right\}\).
\end{enumerate}
\end{proposition}
\begin{proof}
For item (1), the polynomials $f_1$ and $f_2$ satisfy $X(f_1)=2Cf_1$ and $X(f_2)=Af_2$. Moreover, $\mu_1=1$ and $\mu_2=-\dfrac{2C}{A}$ solve $\mu_1(2C)+\mu_2A=0$. By Proposition \ref{prop:firstintegral}, we obtain $P_X(x,y,z)=f_1(x,y,z) f_2(x,y,z)^{-\frac{2C}{A}}$ as a first integral of $X$.

The zero level set of $f_2$ defines the invariant plane \( W^X \), given by \( x - H(Az + y) = 0 \). Intersecting \( W^X \) with \( \Sigma = \{ z = 0 \} \), we obtain the line \( r^X \).

Item (2) follows analogously by replacing \((f_1,f_2,C,A,H,\Lambda,X)\) with \((F_1,F_2,c,a,h,\lambda,Y)\), which yields \eqref{eq:firstintegral_Y} and the line \(r^Y\).
\end{proof}

To characterize the algebraic curves on the switching manifold \(\Sigma\) that support the closing conditions, we start from the first-integral matching equations. 
According to part (a) of Theorem \ref{teo:main1}, a symmetric periodic orbit of the vector field \( Z \) intersects the plane \( \Sigma \) at the two points \( (x,y,0) \) and \( (-y,-x,0) \). Then, the first integrals must take equal values at these intersection points. Thus, \( P_X \) and \( P_Y \) must satisfy the following conditions:
\[
P_X(x,y,0) - P_X(-y,-x,0) = 0,
\]
\[
P_Y(x,y,0) - P_Y(-y,-x,0) = 0.
\]
Using the symmetry hypotheses of the family $\CLequi$, namely \(a=A\), \(c=C\), \(h=H\), and \(\lambda=-\Lambda\), these two conditions reduce to the same algebraic closure relation on \(\Sigma\).

\begin{lemma}\label{lemma:aminus2c}
Consider \( Z \in \CLequi \). If the real eigenvalue of \( DX \) is equal to minus twice the real part of the complex eigenvalue, then the conic
\begin{equation}\label{eq:algebraicurve_A=-2C}
\Gamma_1 := \left\{ (x,y,0) \in \mathbb{R}^3 \ \bigg| \ H(x^2 + y^2) - (H + 1)xy + \dfrac{2CH\Lambda(y - x)}{C^2 + 1} + \dfrac{\Lambda^2(H - 1)}{C^2 + 1} = 0 \right\},
\end{equation}
is an algebraic curve on \( \Sigma \). Furthermore,
\begin{enumerate}
\item[(a)] \( \Gamma_1 \) is a pair of lines if and only if \( H = 1 \);
\item[(b)] \( \Gamma_1 \) is a parabola if and only if \( H = -\dfrac{1}{3} \);
\item[(c)] \( \Gamma_1 \) is a hyperbola if and only if \( -\dfrac{1}{3} < H < 1 \);
\item[(d)] \( \Gamma_1 \) is an ellipse if and only if \( H < -\dfrac{1}{3} \) or \( H > 1 \).
\end{enumerate}
\end{lemma}

\begin{proof}
By the matching conditions for symmetric crossings on \(\Sigma\), the closure equations are
\begin{equation}\label{eq:close_firstintegral_X}
P_X(x,y,0) - P_X(-y,-x,0) = 0,
\end{equation}
\begin{equation}\label{eq:close_firstintegral_Y}
P_Y(x,y,0) - P_Y(-y,-x,0) = 0.
\end{equation}
Under the resonant condition \(A=-2C\), a direct calculation yields
\[
\left(H(x^2 + y^2) - (H + 1)xy + \dfrac{2CH\Lambda(y - x)}{C^2 + 1} + \dfrac{\Lambda^2(H - 1)}{C^2 + 1} \right)(x + y) = 0.
\]
Hence, \(\Gamma_1\) is exactly the conic \eqref{eq:algebraicurve_A=-2C}. Its quadratic-part discriminant is
\[
\Delta_1(H)=(-(H+1))^2-4H\cdot H=-3H^2+2H+1=(1-H)(3H+1).
\]
\begin{itemize}
\item If \(H=1\), then \(\Delta_1(H)=0\) and the constant term also vanishes; thus \(\Gamma_1\) degenerates into a pair of lines.
\item If \(H=-\dfrac{1}{3}\), then \(\Delta_1(H)=0\) and the conic is non-degenerate, hence a parabola.
\item If \(-\dfrac{1}{3}<H<1\), then \(\Delta_1(H)>0\), so \(\Gamma_1\) is a hyperbola.
\item If \(H< -\dfrac{1}{3}\) or \(H>1\), then \(\Delta_1(H)<0\), so \(\Gamma_1\) is an ellipse.
\end{itemize}
\end{proof}

Note that the type of the conic depends on the slope parameter \(H\) associated with the line \(r^X\). In the next analysis, we will determine a critical value of \(H\) by analyzing the first coefficient of a time displacement map. For this critical value, the discriminant is positive, and therefore the curves under consideration are hyperbolas.

\section{Half-Return Times on Algebraic Curves}\label{subsec:hra}
The geometric characterization obtained in the previous section reduces the closing problem to points on the algebraic curves \(\Gamma_1\) and \(\Gamma_2\). We now complement this geometric step with an analytic one: the study of the half-return times associated with \(X\) and \(Y\) along these curves, in a neighborhood of infinity. This will provide the analytic framework used in the next section to define and analyze the time-matching (displacement) function.
\begin{lemma}\label{lema:half_return_delta_X}
Suppose that \( Z = (X, Y) \in \CLequi \). Let \( \Delta^{X}(x_0, y_0) \) be the half-return map associated with the vector field \( X \). Assume additionally that the real eigenvalue of \(DX\) satisfies \(A=-2C\), where \(C\) is the real part of the complex eigenvalues of \(DX\). If \(\Gamma_1\) is a hyperbola and \( (x_0, y_0, 0) \in \Gamma_1 \), then \( t^X(x_0,y_0) \) is analytic on \(\Gamma_1\) in a neighborhood of infinity.

\end{lemma}
\begin{proof}
Under the given assumptions, we have that \( A = -2C \), where \( A \) is the real eigenvalue and \( C \) is the real part of the complex eigenvalue of \( DX \). According to Definition \ref{def:eq_fechamento}, the closing condition for the orbit segment of the vector field \( X \) is given by:
	\begin{equation}\label{eq:matriz_fundamental_superior_ortogonais_X}
		{\rm e}^{DX\,t^{X}}\left( \begin {array}{c}  -2Cx_0+\Lambda H\\ \noalign{\medskip}
		\Lambda\\ \noalign{\medskip}{ y_0}
		\end {array} \right)-\left( \begin {array}{c}  -2Cx_1+\Lambda H\\ \noalign{\medskip}
		\Lambda\\ \noalign{\medskip}{ y_1}
		\end {array} \right)=\left( \begin {array}{c}  0\\ \noalign{\medskip}
		0\\ \noalign{\medskip}0
		\end {array} \right).
	\end{equation}
In this case, the return time depends only on \(y_0\). Without loss of generality, we will parameterize \(\Gamma_1\) as a function of \(y_0\). Therefore, at this stage, it is not necessary to use parametrization of \(\Gamma_1\).  As we are interested in analyzing the return time of orbits near infinity, we introduce the change of variables
\[
t^{X} = \tau^{X} - \pi, \quad y_0 = \left(v_0\right)^{-1}.
\]
Note that the return time \( t^X \) depends only on the second coordinate of the initial condition, \( y_0 \). Thus, if \( y_0 \) tends to infinity, then \( v_0 \) tends to zero and \( t^X \) approaches \( \pi \). Consequently, \( \tau^{X} \) approaches zero. In this context, our first goal is to prove that \( \tau^{X} \) is an analytic function of \( v_0 \) near zero.

Considering the second component of equation \eqref{eq:matriz_fundamental_superior_ortogonais_X} and substituting the new variables into it, we desingularize the expression and obtain the following equation:

\begin{equation}\label{eq:tempo_retorno_x_ortogonal}
F_1\left(v_0,\tau^{X}\right)	=0,
\end{equation}
where $F_1$ represents the given analytical function
$$
F_1\left(v_0,\tau^{X}\right)=\left(\left(\cos \! \left(\tau^{X} \right)-C \sin \! \left(\tau^{X} \right)\right) \Lambda  v_0 -\left(C^{2}+1\right) \sin \! \left(\tau^{X} \right)\right) {\mathrm e}^{C \left(\tau^{X} +\pi \right)}+\Lambda  v_0.
$$
We observe that $F_1\left(0,0\right)=0$ and $\dfrac{\partial F_1}{\partial\tau^{X}}\left(0,0\right)=-\left(C^{2}+1\right){\mathrm e}^{\pi \, C}\neq0$. 

Then, from Implicit Function Theorem \ref{teo:funcao_implicita} we conclude that $\tau^{X}$ is an analytic function of $v_0$. Furthermore, $F_1\left(v_0,\tau^{X}\left(v_0\right)\right)=0$ in a neighborhood of the point $\left(0,0\right)$. 

As we are interest in the Taylor expansion of \(\tau^X\), consider $\tau^{X}=\gamma_{1}^{X}v_0+\gamma_{2}^{X}v_0^2+\cdots$. Then, we have that
	\begin{align*}
		{\rm e}^{C \left( \tau^{X}+\pi \right)}=&{\rm e}^{\pi \, C}\Bigg({\dfrac {C \left( C\left(\gamma_{1}^{X}\right)^{2}+2\,\gamma_{2}^{X} \right) \left(v_0\right)^{2}}{2}}+		C \gamma_{1}^{X}v_0+1\Bigg)+\cdots,\\
		\cos\left(\tau^{X}\right)=&1-{\dfrac {\left(\gamma_{1}^{X}\right)^{2}\left(v_0\right)^{2}}{2}}-+\cdots,\\
		\sin\left(\tau^{X}\right)=&\gamma_{1}^{X} v_0+\gamma_{2}^{X}\left(v_0\right)^{2}+\cdots.
	\end{align*}
	Substituting these expressions in equation \eqref{eq:tempo_retorno_x_ortogonal},
	we obtain
	$$
	M_1v_0+M_2\left(v_0\right)^2+\cdots=0,
	$$
	where	$
	M_1=\left(\dfrac{1+{{\rm e}^{-\pi\,C}}}{{C}^{2}+1}\right) \Lambda-\gamma_{1}^{X}
$
and
$
	M_2=C\left(\gamma_{1}^{X}\right)^{2}+\gamma_{2}^{X}.
	$

Thus, given that $M_1=M_2=0,$ we conclude that
\begin{equation}\label{eq:coef_tx1}
\gamma_{1}^{X}=\dfrac{\left(1+{{\rm e}^{-\pi\, C}} \right) \Lambda}{{C}^{2}+1}
\end{equation}
and
\begin{equation}\label{eq:coef_tx2}
\gamma_{2}^{X}=-C\left(\dfrac{\left(1+{{\rm e}^{-\pi\, C}} \right) \Lambda}{{C}^{2}+1}\right)^2.
\end{equation}
\end{proof}

\begin{lemma}\label{lema:half_return_delta_Y}
Suppose that \( Z = (X, Y) \in \CLequi \). Let \( \Delta^{Y}(x_0, y_0) \) denote the half-return map associated with the vector field \( Y \). Assume additionally that the real eigenvalue of \(DX\) satisfies \(A=-2C\), where \(C\) is the real part of the complex eigenvalues of \(DX\). If \(\Gamma_1\) is a hyperbola and \( (x_0, y_0, 0) \in \Gamma_1 \), then the return time \( t^Y(x_0, y_0) \) is analytic on \(\Gamma_1\) in a neighborhood of infinity.
\end{lemma}
\begin{proof}
Since \(Z\in\CLequi\), Proposition~\ref{prop:equivariant} yields \(c=C\), \(a=A\), \(h=H\), and \(\lambda=-\Lambda\). Together with the spectral assumption \(A=-2C\), we obtain \(a=-2C\).

In this case, the return time depends on \(x_0\). Since \(\Gamma_1\) is parameterized as a function of \(y_0\), we will use a parametrization of one branch of \(\Gamma_1\). We observe that the discriminant of the conic \(\Gamma_1\), denoted by \(D(\Gamma_1)\), is positive. Without loss of generality, we assume that \(0 < H < 1\) and restrict our analysis to the branch of the hyperbola \(\Gamma_1\) described by:
\[
x = \dfrac{(H + 1) y}{2H} + \dfrac{\Lambda C}{C^{2}+1} + \sqrt{\dfrac{ \Lambda \left((C^{2}+1)(\Lambda + C(1 - H)y) - \Lambda H \right)}{ H (C^{2}+1)^2} + \dfrac{D(\Gamma_1)}{4H^2}},
\]
on which the initial point will be chosen for the subsequent analysis, since the possible periodic orbits of our system, given by~\eqref{eq:canonica_ortogonal} under the hypotheses of our lemma, must intersect the switching manifold \(\Sigma\) on \(\Gamma_{1}\).

The conic \( \Gamma_1 = 0 \) intersects the coordinate axes at the following points:
\[
\left( \frac{(C H + \sqrt{H(C^2 + 1 - H)}) \Lambda}{H(C^2 + 1)},\, 0 \right), \quad
\left( \frac{(C H - \sqrt{H(C^2 + 1 - H)}) \Lambda}{H(C^2 + 1)},\, 0 \right),
\]
\[
\left( 0,\ -\frac{(C H - \sqrt{H(C^2 + 1 - H)}) \Lambda}{H(C^2 + 1)} \right), \quad
\left( 0,\ -\frac{(C H + \sqrt{H(C^2 + 1 - H)}) \Lambda}{H(C^2 + 1)} \right).
\]

These intercepts are consistent with the branch selection adopted above and justify
choosing initial points in the first quadrant, corresponding to the branch of
\( \Gamma_1 \) considered as \( y \to +\infty \) (see Fig. \ref{fig:hiperbole}).
The asymptotic direction of this branch is computed later through the limit
\(y_0/x_0(y_0)\), so we do not use it at this stage. Furthermore, the hyperbola does
not intersect the lines \( r^X \) and \( r^Y \), since these lines are contained in the
invariant linear manifolds associated with the respective vector field.
\begin{figure}[h!]	
\begin{center}
	\begin{overpic}[scale=0.09]{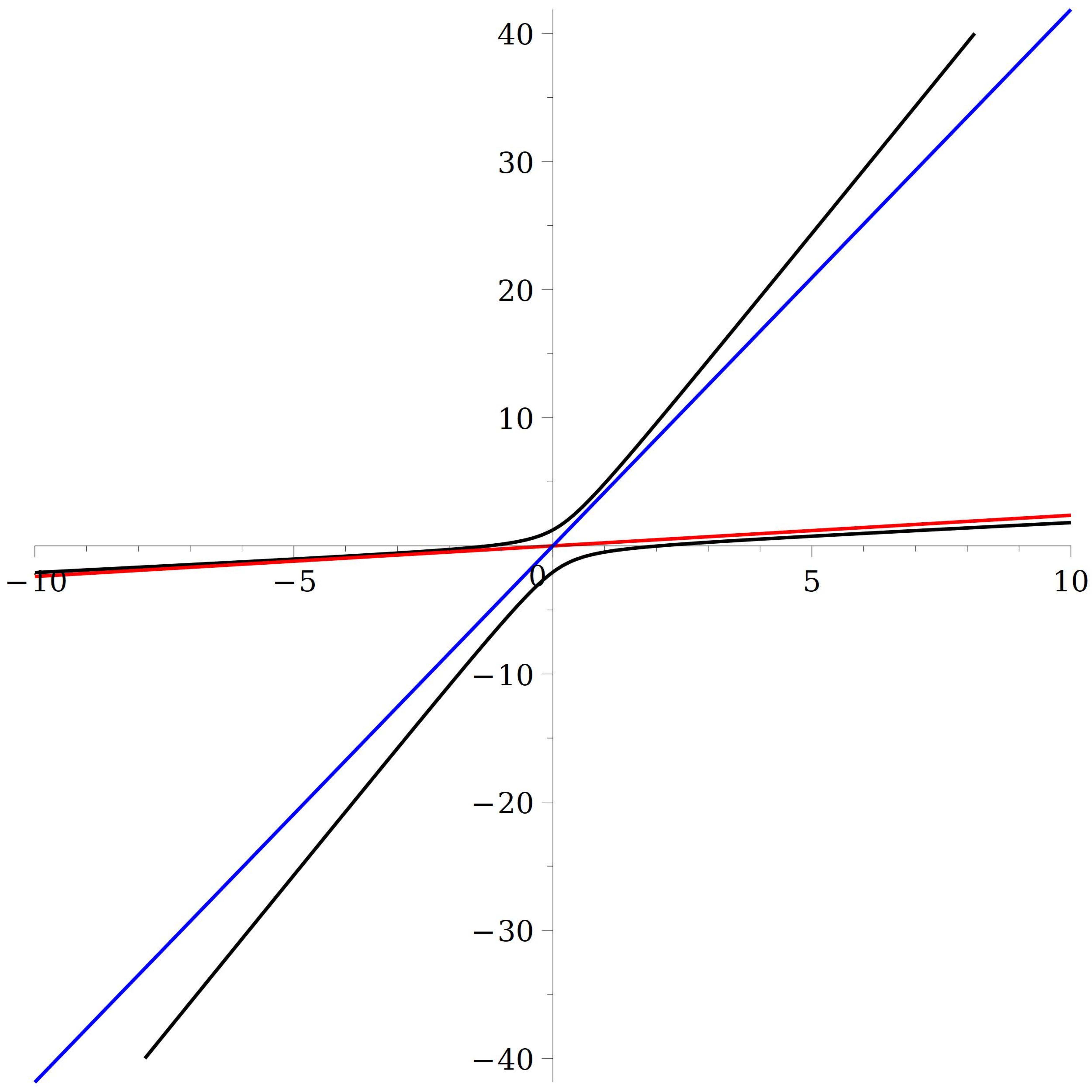}
\put(90,85){\textcolor{blue}{$r^{X}$}}
\put(90,55){\textcolor{red}{$r^{Y}$}}
	\put(92,42){$x$}
	\put(42,92){$y$}
	\end{overpic}
	\end{center}
		\caption{Representation of the hyperbola \( \Gamma_1 = 0 \) with the lines \(r^{X}\) and \(r^{Y}\).}
\label{fig:hiperbole}
	\end{figure}
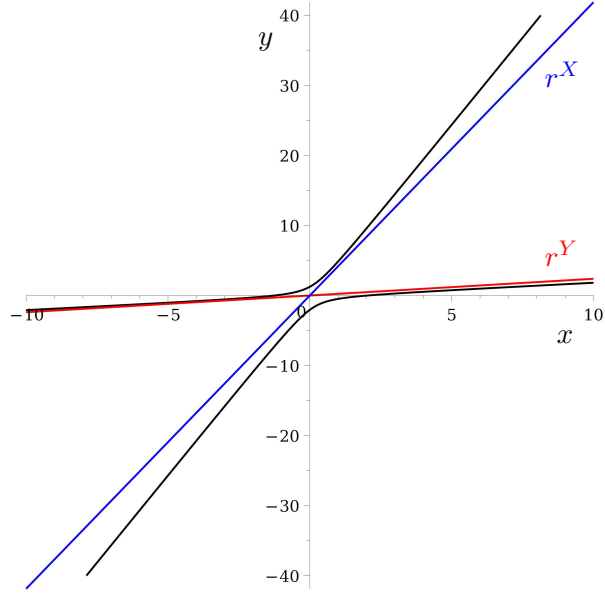

The methodology employed in the previous proof will be replicated step by step for the bottom half-space $\Sigma^{-}$, starting with the equation
	\begin{equation}\label{eq:matriz_fundamental_superior_ortogonais_Y}
		{\rm e}^{DY\,\left(-t^{Y}\right)}
\left( 
\begin {array}{c}  
	- \Lambda\\ 
\noalign{\medskip}
	-2 C y_0-\Lambda H\\ 
\noalign{\medskip}
	x_0
\end {array} 
\right)
-
\left( 
\begin {array}{c}  
	- \Lambda\\ 
\noalign{\medskip}
	-2 C y_2-\Lambda H\\ 
\noalign{\medskip}
	x_2
\end {array} 
\right)
=\left( \begin {array}{c}  0\\ \noalign{\medskip}
		0\\ \noalign{\medskip}0
		\end {array} \right).
	\end{equation}
with the following transformations:
\[
y_0 = \left(v_0\right)^{-1}, \quad \tau^{Y} = -\left(t^{Y} + \pi\right), 
\]
\[
x_0 = \dfrac{H + 1}{2H v_0} + \dfrac{\Lambda C}{C^{2} + 1} + \sqrt{\dfrac{\Lambda \left((C^{2} + 1) \left(\Lambda + \dfrac{C (1 - H)}{v_0} \right) - \Lambda H \right)}{H (C^{2} + 1)^2} + \dfrac{D\left(\Gamma_1\right)}{4 H^2 (v_0)^2}},
\]
where \(D\! \left(\Gamma_{1}\right)\) is the discriminant of the hyperbole \(\Gamma_{1}\).

Note that the return time \( t^Y \) originally depends only on the first coordinate of the initial condition, \( x_0 \). After the coordinate change used to select a point on the hyperbola, the return time becomes a function of the new variable \( v_0 \). In particular, as \( v_0 \) tends to zero, \( t^Y \) approaches \( -\pi \). Consequently, \( \tau^{Y}  \) tends to zero.

Considering the first component of equation \eqref{eq:matriz_fundamental_superior_ortogonais_Y} and substituting the new variables into it, we desingularize the expression and obtain the following equation:
\begin{equation}\label{eq:tempo_retorno_y_ortogonal}
F_2\left(v_0,\tau^{Y}\right)	=0,
\end{equation}

where \(F_2\) represents the given analytical function

\begin{align*}
F_2(v_0,\tau^Y)=&\Bigg(\left(C^{2}+1\right) \Bigg(1+\Bigg(\sqrt{\frac{4 \left(C^{2}-H +1\right) \left(\Lambda v_0\right)^{2}}{\left(H \left(C^{2}+1\right)\right)^{2}}+\frac{4 C \Lambda  \left(1-H\right) v_0}{\left(C^{2}+1\right) H}+\frac{D\! \left(\Gamma_{1}\right)}{H^{2}}}\\
&+1\Bigg) H \Bigg) \sin \! \left(\tau^{Y}\right)-2 \Lambda  v_0 \cos \! \left(\tau^{Y}\right) H \Bigg) {\mathrm e}^{-C \left(\tau^{Y}+\pi \right)}-2 H \Lambda  v_0.
\end{align*}
We observe that $F_2\left(0,0\right)=0$ and
\[
\dfrac{\partial F_2}{\partial t^Y}\left(0,0\right)=\frac{\left(C^{2}+1\right) \left(\sqrt{\dfrac{D(\Gamma_{1})}{H^{2}}}\, H +H +1\right) {\mathrm e}^{-C \pi}}{2 H}.
\]
As \(\dfrac{\partial F_2}{\partial t^Y}(0,0) \neq 0\) for \(H \neq 0\), it follows from the Implicit Function Theorem \ref{teo:funcao_implicita} that \(t^Y\) is an analytic function of \(v_0\). 
Therefore, $\tau^Y\left(v_0\right)=\displaystyle\sum_{n\geq1}\gamma^{Y}_{n}\left(v_0\right)^{n}$ defines an analytic function on \(\Gamma_1\) for \(v_0\) near zero. 

Straightforward calculations show that
\begin{equation}\label{eq:coef_ty1}
\gamma_{1}^Y=\frac{2 H \Lambda  \left({\mathrm e}^{C \pi}+1\right)}{\left(C^{2}+1\right) \left(\sqrt{\dfrac{D\left(\Gamma_{1}\right)}{H^{2}}}\, H +H +1\right)}
\end{equation}
and
\begin{equation}\label{eq:coef_ty2}
\gamma_{2}^Y=-\frac{2 C \,H^{2} \Lambda^{2} \left({\mathrm e}^{C \pi}+1\right) \left(\sqrt{\dfrac{D\left(\Gamma_{1}\right)}{H^{2}}}\, H -3 \left(3 H +1\right) {\mathrm e}^{C \pi}\right)}{\left(C^{2}+1\right)^{2} \left(3 H +1\right) \left(\left(H^{2}+H \right) \sqrt{\dfrac{D\left(\Gamma_{1}\right)}{H^{2}}}-H^{2}+2 H +1\right)}.
\end{equation}

\end{proof}
\section{The Time-Matching Function}\label{sec:time_matching}

In this section, we investigate the difference between the return times \(t^X\) and \(t^Y\) associated with the half-return Poincaré maps of the vector fields \(X\) and \(Y\), respectively. The periodic orbit, when it exists, intersects the switching manifold \(\Sigma\) at two points lying on the conic \(\Gamma\) given by Equation \ref{eq:algebraicurve_A=-2C}. Due to the T-symmetry of the periodic orbit, the condition for periodicity reduces to the vanishing of the difference between these return times. This observation motivates the definition of a time-matching function, whose zeros characterize periodic behavior in the system.

\begin{definition}\label{def:time_matching}
Assume the hypotheses of Lemmas \ref{lema:half_return_delta_X} and \ref{lema:half_return_delta_Y}. Let \( \tau^{X}(v_0) \) and \( \tau^{Y}(v_0) \) denote the flight times at infinity along an interval \(0\in I\subset \Gamma_1\cup\Gamma_2\) of an orbit under the vector fields \( X \) and \( Y \), respectively. We define the \emph{time-matching function} \( \tau : I \to \mathbb{R} \) by
\begin{equation}\label{eq:time_matching}
\tau(v_0) = \tau^{X}(v_0) - \tau^{Y}(v_0).
\end{equation}
In particular, we denote the first two coefficients of the Taylor expansion of \( \tau(v_0) \) by \( \gamma_1 = \gamma_{1}^X - \gamma_{1}^Y \) and \( \gamma_2 = \gamma_{2}^X - \gamma_{2}^Y \), where \(\gamma_{1}^X\), \(\gamma_{2}^X \), \(\gamma_{1}^Y\) and \(\gamma_{2}^Y\) are given by \eqref{eq:coef_tx1}, \eqref{eq:coef_tx2}, \eqref{eq:coef_ty1}, \eqref{eq:coef_ty2}, respectively. 
\end{definition}
With this definition, we are now ready to prove the final statement of our main result.
\begin{proof}[Proof of part (b) of Theorem \ref{teo:main1}]

Suppose \(A=-2C\) and consider the critical point given by 
\[
\overline{H} = \frac{1}{2 \cosh \! \left(\pi C  \right)-1},
\] 
which satisfies the inequality
\[
-\dfrac{1}{3} < \overline{H} < 1,
\]
for all \(C \neq 0\). Therefore, under this condition, the discriminant \(D(\Gamma_1)\) is positive, confirming that \(\Gamma_1\) is a hyperbola.
Substituting this value into the coefficients \(\gamma_{1}\) and \(\gamma_{2}\), given by Definition \ref{def:time_matching}, we conclude that
\[
\gamma_{1}=0
\]
and 

\[
\gamma_{2} = 
\begin{cases}
\displaystyle -\frac{C \,\Lambda^{2} \left(1+2 \,{\mathrm e}^{-C \pi}-{\mathrm e}^{2 C \pi}+{\mathrm e}^{-2 C \pi}+{\mathrm e}^{4 C \pi}+2 \,{\mathrm e}^{3 C \pi}\right)}{\left(C^{2}+1\right)^{2}}, & \text{if } C < 0, \\[1em]
\displaystyle -\frac{2 \Lambda^{2} C \left({\mathrm e}^{-C \pi}+1+{\mathrm e}^{-2 C \pi}\right)}{\left(C^{2}+1\right)^{2}}, & \text{if } C > 0.
\end{cases}
\]
For \( C < 0 \), we aim to prove that the equation \( \gamma_{2} = 0 \) has no solution. By multiplying both sides of the equation by \( \mathrm{e}^{2C \pi} \) and setting \( \mathrm{e}^{C \pi} = x \), we obtain the equivalent equation \( p(x) = 0 \), where
\[
p(x) = x^6 + 2x^5 - x^4 + x^2 + 2x + 1.
\]
Our goal is to show that \( p(x) \) has no positive real roots, which implies \( \gamma_2 \neq 0 \). The roots of $p(x)$ are approximately given by

\[
x\approx-2.39314; -0.567069; -0.384611 \pm 0.681543 i; 0.864713 \pm 0.674897 i.
\]



Hence, the polynomial \( p(x) \) has no positive real roots, and as a result,
\[
\gamma_2 \neq 0 \quad \text{for all } C \neq 0.
\]

By applying the Weierstrass Preparation Theorem \ref{teo:teorema_preparacao_weierstrass}, we conclude that there exist analytic functions \( \delta_1(H) \) and \( \overline{\tau}(v_0, H) \neq 0 \), defined in a neighborhood \( U_0 \) of \( \overline{H} \). Consequently, the time-matching function \( \tau(v, H) \) admits the factorization:
\[
\tau(v, H) = (\delta_1(H) + v)\, \overline{\tau}(v, H).
\]

Under these conditions, \( \tau(v, H) \) has at most one positive zero, which corresponds to the existence of at most one time symmetric limit cycle. This zero is isolated due to the analyticity of the function.
\end{proof}

\section{Stability of limit cycle}
\label{subsec:stability_saltation}

For a symmetric limit cycle, one has \(t_X=t_Y=T/2\). Substituting \eqref{eq:fundamental_conjugacy} into \eqref{eq:monodromy_saltation}, we obtain
\begin{equation}\label{eq:monodromy_equivariant_reduction}
	M
	=
	S_{Y\to X}(p_0)\,[S]\,\Phi_X(T/2)\,[S]\,
	S_{X\to Y}(p_1)\,\Phi_X(T/2).
\end{equation}
Hence the spectral analysis of \(M\) reduces to a single half-return variational flow \(\Phi_X(T/2)\), together with the saltation corrections at the two crossings.

To specialize this identity to the symmetric limit cycle under the hypothesis \(A=-2C\), we evaluate the saltation matrices at the crossing points
\[
p_0=(x_0,y_0,0),
\qquad
p_1=(-y_0,-x_0,0).
\]

\begin{corollary}\label{cor:detM_geometry_A_minus_2C}
Assume that the real eigenvalue is minus twice the real part of the complex
conjugate eigenvalues. At the periodic-orbit crossing points
\(p_0=(x_0,y_0,0)\) and \(p_1=(-y_0,-x_0,0)\) of the symmetric cycle, one has
\begin{equation}\label{eq:detM_y0x0}
\det(M)=\left(\frac{y_0}{x_0}\right)^2.
\end{equation}
\end{corollary}

\begin{proof}
In the parameter regime \(A=-2C\), one has
\(\operatorname{tr}(DX)=A+2C=0\), so
\[
\det\!\bigl(\Phi_X(T/2)\bigr)=e^{\operatorname{tr}(DX)T/2}=1,
\qquad
\det\!\bigl(\Phi_Y(T/2)\bigr)=1.
\]
Consequently, the determinant of \(M\) satisfies
\begin{equation}\label{eq:det_monodromy_saltation}
\det(M)=\det\!\bigl(S_{Y\to X}(p_0)\bigr)\det\!\bigl(S_{X\to Y}(p_1)\bigr),
\qquad
\det\!\bigl(S_{\alpha\to\beta}(p)\bigr)=
\frac{\nabla f(p)^T\beta(p)}{\nabla f(p)^T\alpha(p)}.
\end{equation}

Since \(f(x,y,z)=z\), we have \(\nabla f=(0,0,1)^T\), so
\(\nabla f(p)^T F\) is the \(z\)-component of the corresponding vector field.
From the canonical form~\eqref{eq:canonica_ortogonal}, at \(z=0\),
\[
X_3(x,y,0)=y,
\qquad
Y_3(x,y,0)=x.
\]
At \(p_0\), the crossing is from \(Y\) to \(X\), hence
\[
\det\!\bigl(S_{Y\to X}(p_0)\bigr)=
\frac{\nabla f(p_0)^T X(p_0)}{\nabla f(p_0)^T Y(p_0)}
=
\frac{y_0}{x_0}.
\]
At \(p_1=(-y_0,-x_0,0)\), the crossing is from \(X\) to \(Y\), hence
\[
\det\!\bigl(S_{X\to Y}(p_1)\bigr)=
\frac{\nabla f(p_1)^T Y(p_1)}{\nabla f(p_1)^T X(p_1)}
=
\frac{-y_0}{-x_0}
=
\frac{y_0}{x_0}.
\]
Substituting these two identities into \eqref{eq:det_monodromy_saltation}
gives \eqref{eq:detM_y0x0}.
\end{proof}

We now state a result for the whole symmetric class \(Z\in\CLequi\), expressing the
transverse quadratic factor of the characteristic polynomial in terms of matrix invariants.

\begin{lemma}[Transverse characteristic polynomial]\label{lem:transverse_quadratic}
Consider a crossing periodic orbit of a vector field \(Z\in\CLequi\), and let \(M\)
be its saltation-corrected monodromy matrix. In this symmetric class, the
characteristic polynomial factors as \(p_M(\mu)=(\mu-1)q(\mu)\), where
\(q\) is the transverse quadratic polynomial
\[
q(\mu)=\mu^2-(\operatorname{tr}(M)-1)\mu+\det(M).
\]
\end{lemma}

\begin{proof}
Because the system is autonomous, the tangent direction to the periodic orbit is
preserved after one period, hence \(M\) has the trivial Floquet multiplier
\(\mu_1=1\). If \(\mu_2\) and \(\mu_3\) are the nontrivial multipliers, then the
characteristic polynomial of \(M\) is
\[
p_M(\mu)=\det(\mu I-M)=(\mu-1)(\mu-\mu_2)(\mu-\mu_3)
=(\mu-1)q(\mu),
\]
with
\[
q(\mu)=\mu^2-(\mu_2+\mu_3)\mu+\mu_2\mu_3.
\]
Using
\[
\operatorname{tr}(M)=1+\mu_2+\mu_3,
\qquad
\det(M)=\mu_1\mu_2\mu_3=\mu_2\mu_3,
\]
we obtain
\[
\mu_2+\mu_3=\operatorname{tr}(M)-1,
\qquad
\mu_2\mu_3=\det(M).
\]
Substituting these identities into the expression of \(q\) gives
\[
q(\mu)=\mu^2-(\operatorname{tr}(M)-1)\mu+\det(M).
\]
\end{proof}

Since one Floquet multiplier of an autonomous periodic orbit is always equal to \(1\),
orbital stability is determined by the two nontrivial, or transverse, multipliers.
Therefore, to guarantee that the cycle is attracting, we seek conditions ensuring that
both of them lie in the unit disk. To this end, we apply the Schur--Cohn criterion to
the quadratic polynomial associated with the transverse dynamics.
More precisely, for a monic quadratic polynomial
\[
r(\mu)=\mu^2+a_1\mu+a_0,
\]
with real coefficients, all roots satisfy \(|\mu|<1\) if and only if
\[
1+a_1+a_0>0,
\qquad
1-a_1+a_0>0,
\qquad
1-a_0>0.
\]
In our setting,
\[
a_1=-(\operatorname{tr}(M)-1),
\qquad
a_0=\det(M),
\]
and therefore the Schur--Cohn criterion yields the inequalities stated below.

\begin{proposition}[Schur--Cohn conditions for transverse multipliers]
\label{prop:schur_conditions}
Consider a crossing periodic orbit of a vector field \(Z\in\CLequi\), and let \(M\)
be its saltation-corrected monodromy matrix. Let \(q\) be the transverse quadratic
polynomial from Lemma~\ref{lem:transverse_quadratic}. Then \(|\mu_2|<1\) and
\(|\mu_3|<1\) are guaranteed by the Schur--Cohn conditions
\[
1-\det(M)>0,
\qquad
2-\operatorname{tr}(M)+\det(M)>0,
\qquad
\operatorname{tr}(M)+\det(M)>0.
\]
\end{proposition}

\begin{proof}
The factorization \(p_M(\mu)=(\mu-1)q(\mu)\) and the identities for \(q\)
follow from Lemma~\ref{lem:transverse_quadratic}. The three inequalities are
exactly the Schur criterion for a monic quadratic polynomial, applied to \(q(\mu)\).
\end{proof}

We now analyze the asymptotic behavior of the matrix \(M\) along the branch
\(\Gamma\) at infinity, which is precisely the regime where the existence of the
symmetric cycle was proved. This asymptotic analysis constitutes the first
step toward the spectral stability of the cycle.

Corollary~\ref{cor:detM_geometry_A_minus_2C} expresses \(\det(M)\) in terms of the
crossing coordinates in the case \(A=-2C\). In particular, our first objective is
to prove that \(\det(M)<1\) for \(y_0\) sufficiently large, since this is one of the
conditions needed to ensure that the transverse polynomial satisfies the
Schur--Cohn conditions stated in Proposition~\ref{prop:schur_conditions}.

\begin{proposition}\label{prop:m-limit-hyperbola}
Assume that \(Z\in\CLequi\), that \(A=-2C\), and that \(\Gamma_1\) is a hyperbola.
Let \(\Gamma\) be the branch of \(\Gamma_1\) used to parametrize the initial condition
\((x_0,y_0,0)\) of the symmetric cycle, and write \(x_0=x_0(y_0)\). If we define

\begin{equation}\label{eq:m_A_-2C}
m_{\Gamma_{1}}:= \frac{2H}{H+\sqrt{D(\Gamma_1)}+1},
\end{equation}

then
\[
\lim_{y_0\to\infty}\det(M)
=
\left(m_{\Gamma_{1}}\right)^2.
\]
Furthermore, 
\[
\left(m_{\Gamma_{1}}\right)^2
<1.
\]
Consequently, there exists \(\nu>0\) such that \(\det(M)<1\) for every \(y_0>\nu\). Moreover,
\[
\det(M)
=
\left(m_{\Gamma_{1}}\right)^2+o(1)
\qquad\text{as } y_0\to\infty.
\]
\end{proposition}

\begin{proof}
Since \(\Gamma_1\) is a hyperbola, its discriminant satisfies \(D(\Gamma_1)>0\), and by the
classification of \(\Gamma_1\) one has \(-\frac13<H<1\). Therefore, by the asymptotic geometry of
the chosen branch \(\Gamma\) at infinity,
\[
m_{\Gamma_{1}}
=
\frac{2H}{H+\sqrt{D(\Gamma_1)}+1}=\lim_{y_0\to\infty}\frac{y_0}{x_0(y_0)}.
\]

Using \eqref{eq:detM_y0x0}, we obtain
\[
\lim_{y_0\to\infty}\det(M)
=
\left(\lim_{y_0\to\infty}\frac{y_0}{x_0(y_0)}\right)^2
=
\left(m_{\Gamma_{1}}\right)^2.
\]
Moreover,
\[
H+\sqrt{D(\Gamma_1)}+1>H+1>\frac23>0,
\]
so the denominator is positive. We now prove that \(|m_{\Gamma_{1}}|<1\).

If \(0<H<1\), then
\[
2H<H+1<H+\sqrt{D(\Gamma_1)}+1,
\]
and hence \(0<m<1\). If \(H=0\), then \(m_{\Gamma_{1}}=0\). Finally, if \(-\frac13<H<0\), then
\[
|m_{\Gamma_{1}}|=\frac{-2H}{H+\sqrt{D(\Gamma_1)}+1},
\]
and it is enough to show that
\[
-2H<H+\sqrt{D(\Gamma_1)}+1,
\]
or equivalently,
\[
-3H<\sqrt{D(\Gamma_1)}+1.
\]
Since \(H>-\frac13\), we have \(-3H<1\), and since \(D(\Gamma_1)>0\), we also have
\(1<1+\sqrt{D(\Gamma_1)}\). Thus \(-3H<1+\sqrt{D(\Gamma_1)}\), proving \(|m_{\Gamma_{1}}|<1\). Consequently,
\(\left(m_{\Gamma_{1}}\right)^2<1\).

Finally, since \(\lim_{y_0\to\infty}\det(M)=\left(m_{\Gamma_{1}}\right)^2<1\), there exists \(\nu>0\) such that
\(\det(M)<1\) for every \(y_0>\nu\). The asymptotic expansion
\[
\det(M)
=
\left(m_{\Gamma_{1}}\right)^2+o(1)
\qquad\text{as } y_0\to\infty
\]
is just another way of writing the same limit.
\end{proof}

We now consider the analogous asymptotic regime for \(A=2C\). In this case, the relevant conic is
\(\Gamma_2\), whose discriminant is
\[
D(\Gamma_2)=H^2+2H-3.
\]

Having established the asymptotic behavior of the transverse determinant for $A = -2C$, we now turn our attention to the trace of the monodromy matrix. As dictated by the Schur--Cohn conditions (Proposition~\ref{prop:schur_conditions}), the trace is the remaining invariant required to fully determine the stability of the transverse multipliers. 

\begin{proposition}\label{prop:trace-limit-hyperbola}
Assume that $Z\in\CLequi$, that $A=-2C$, and that $0<H<1$, so that $\Gamma_1$ is a hyperbola. Let $\Gamma$ be the branch of $\Gamma_1$ used to parametrize the initial condition $(x_0, y_0, 0)$ of the symmetric limit cycle. If we define
\begin{equation}\label{eq:trace-asymptotic}
\tau_{\Gamma_{1}} := \left(m_{\Gamma_{1}}\right)^2 \left( \Psi_2 e^{2\pi C} + \Psi_{-1} e^{-\pi C} + \Psi_{-4} e^{-4\pi C} \right),
\end{equation}
where $m_{\Gamma_{1}}$ is given by \eqref{eq:m_A_-2C} and the algebraic coefficients $\Psi_k(H)$ are defined as:
\begin{align*}
\Psi_2 &= -\frac{\left(H +1\right) \left(H -\sqrt{D\! \left(\Gamma_{1}\right)}-3\right)}{2}, \\
\Psi_{-1} &= \frac{\left(H^{2}-1\right) \left(\left(H +1\right) \sqrt{D\! \left(\Gamma_{1}\right)}-\left(H -1\right)^{2}+2\right)}{H^{2}}, \\
\Psi_{-4} &= \frac{\left(H +1\right) \sqrt{D\! \left(\Gamma_{1}\right)}-\left(H -1\right)^{2}+2}{2},
\end{align*}
then, the asymptotic trace of the monodromy matrix $M$ as $y_0 \to \infty$ is given by
\begin{equation}\label{eq:trace-asymptotic-limit}
 \lim_{y_0 \to \infty} \text{tr}(M) = \tau_{\Gamma_{1}}.
\end{equation}

\end{proposition}

\begin{proof}
By parametrizing the initial condition along the branch $\Gamma$ 
\[
x = \dfrac{(H + 1) y}{2H} + \dfrac{\Lambda C}{C^{2}+1} + \sqrt{\dfrac{ \Lambda \left((C^{2}+1)(\Lambda + C(1 - H)y) - \Lambda H \right)}{ H (C^{2}+1)^2} + \dfrac{D(\Gamma_1)}{4H^2}},
\] 
with respect to the regularizing parameter $y= 1/v_0$, we take the simultaneous limit as $v_0 \to 0$ and the time of flight $t \to \pi$. The evaluation of the trace of the transition matrices yields the explicit rational expression. Factoring out $\left(m_{\Gamma_{1}}\right)^2$ we obtain \eqref{eq:trace-asymptotic-limit}.
\end{proof}

To complete the proof of item~(c) of Theorem~\ref{teo:main1}, we now pass from the asymptotic formulas for the monodromy invariants to a parameter description of the stable regime. The key point is that the Schur--Cohn inequalities admit a well-defined limit along the large-amplitude symmetric branch, and this limit selects an open subset of the \((C,H)\)-plane where the corresponding crossing symmetric cycle is orbitally asymptotically stable.

\begin{theorem}[Asymptotic description of the stability region in the \((C,H)\)-plane]\label{teo:stability-band}
Assume the hypotheses of Proposition~\ref{prop:trace-limit-hyperbola}, \(C>0\), and the continuation setting of item~(b) of Theorem~\ref{teo:main1}. Then the asymptotic Schur--Cohn conditions for the two transverse Floquet multipliers reduce to
\begin{equation}\label{eq:stability-band}
\begin{cases}
2 + \left(m_{\Gamma_{1}}\right)^2 - \tau_{\Gamma_{1}} > 0,\\
\tau_{\Gamma_{1}} + \left(m_{\Gamma_{1}}\right)^2 > 0.
\end{cases}
\end{equation}
These inequalities define an open subset of the \((C,H)\)-plane, denoted by
\[
\widetilde{\mathcal B}_{\mathrm{st}},
\]
which we call the asymptotic stability band; see Figure~\ref{fig:stability-region}.

For every \((C,H)\in\widetilde{\mathcal B}_{\mathrm{st}}\), the large-amplitude tail of the corresponding symmetric branch is asymptotically stable. Consequently, there exists an open subset
\[
\mathcal B_{\mathrm{st}}\subset \mathbb R^2_{(C,H)},
\]
called the stability region, containing this asymptotically stable tail, such that the corresponding symmetric limit cycle is asymptotically stable for every \((C,H)\in\mathcal B_{\mathrm{st}}\).

Moreover, the boundary of \(\widetilde{\mathcal B}_{\mathrm{st}}\) corresponds asymptotically to two codimension-one losses of stability:
\begin{enumerate}
    \item[(i)] the upper boundary
    \[
    2 + \left(m_{\Gamma_{1}}\right)^2 - \tau_{\Gamma_{1}} = 0
    \]
    coincides with
    \[
    H_{\mathrm{crit}}=\frac{1}{2\cosh(\pi C)-1},
    \]
    where a transverse multiplier approaches \(1\);
    \item[(ii)] the lower boundary
    \[
    \tau_{\Gamma_{1}} + \left(m_{\Gamma_{1}}\right)^2 = 0
    \]
    corresponds to a transverse multiplier approaching \(-1\).
\end{enumerate}
\end{theorem}

\begin{figure}[h!]
	\centering
	\includegraphics[width=0.45\textwidth]{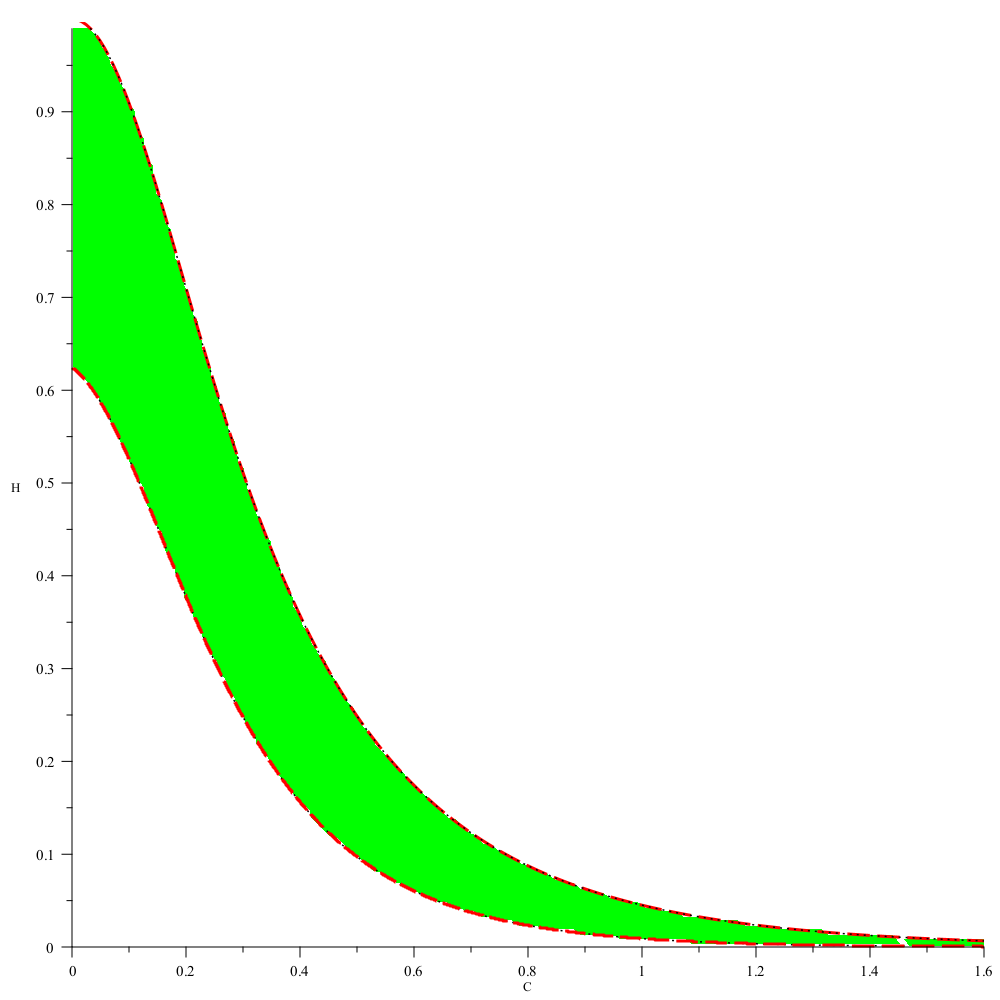}
	\caption{Stability band in the \((C,H)\)-plane for \(A=-2C\). The green region corresponds to the open set \(\mathcal B_{\mathrm{st}}\) where the two nontrivial transverse Floquet multipliers lie inside the unit disk. The dashed boundary curves are defined asymptotically by
		\(
		\tau_{\Gamma_{1}} + \left(m_{\Gamma_{1}}\right)^2 = 0
		\qquad\text{and}\qquad
		\tau_{\Gamma_{1}} = 2 + \left(m_{\Gamma_{1}}\right)^2.
		\)}
	\label{fig:stability-region}
\end{figure}

\begin{proof}
By Proposition~\ref{prop:m-limit-hyperbola},
\[
\det(M)\to \left(m_{\Gamma_{1}}\right)^2<1
\qquad\text{as } y_0\to\infty.
\]
Therefore, along the large-amplitude tail of the symmetric branch, the first Schur--Cohn inequality
\[
1-\det(M)>0
\]
holds for sufficiently large \(y_0\). This identifies an asymptotically stable tail, but it does not by itself describe the whole stability region, since stable symmetric cycles may also occur at finite and small amplitudes.

The remaining two Schur--Cohn inequalities are
\[
2-\operatorname{tr}(M)+\det(M)>0
\qquad\text{and}\qquad
\operatorname{tr}(M)+\det(M)>0.
\]
Using Propositions~\ref{prop:m-limit-hyperbola} and \ref{prop:trace-limit-hyperbola}, these converge along the large-amplitude tail to
\[
2 + \left(m_{\Gamma_{1}}\right)^2 - \tau_{\Gamma_{1}} >0
\qquad\text{and}\qquad
\tau_{\Gamma_{1}} + \left(m_{\Gamma_{1}}\right)^2 >0.
\]
Hence \eqref{eq:stability-band} defines an open subset
\[
\widetilde{\mathcal B}_{\mathrm{st}}\subset \mathbb R^2_{(C,H)},
\]
and every sufficiently large-amplitude symmetric limit cycle with parameters in \(\widetilde{\mathcal B}_{\mathrm{st}}\) is asymptotically stable.

Now let \(\gamma_{C,H}\) denote the symmetric branch selected in item~(b) of Theorem~\ref{teo:main1}. By the closing relation obtained through the Weierstrass Preparation Theorem, this branch depends continuously on the parameters as long as the crossing configuration persists. Consequently, the associated saltation-corrected monodromy matrix \(M\), its trace and determinant, and the two transverse Floquet multipliers also vary continuously along the branch.

Since the open unit disk is an open subset of the complex plane, asymptotic stability is an open property along the continued branch. Therefore each asymptotically stable large-amplitude cycle belonging to the tail determined above is surrounded by an open neighborhood of parameter values for which the continued symmetric cycle remains asymptotically stable. Taking the union of all such neighborhoods, we obtain an open set
\[
\mathcal B_{\mathrm{st}},
\]
which contains the stable large-amplitude tail and in which the corresponding symmetric limit cycle is asymptotically stable.

This also explains the occurrence of finite- and small-amplitude cycles inside the stability region: they belong to the same continued symmetric branch, and their stability follows by continuation from the asymptotically stable tail. In particular, asymptotic stability persists under sufficiently small perturbations of \(C\) and \(H\) within \(\mathcal B_{\mathrm{st}}\).

Finally, on the upper asymptotic boundary
\[
2 + \left(m_{\Gamma_{1}}\right)^2 - \tau_{\Gamma_{1}} = 0,
\]
one transverse multiplier approaches \(1\), recovering the critical value
\[
H_{\mathrm{crit}}=\frac{1}{2\cosh(\pi C)-1}.
\]
On the lower asymptotic boundary
\[
\tau_{\Gamma_{1}} + \left(m_{\Gamma_{1}}\right)^2 = 0,
\]
one transverse multiplier approaches \(-1\), corresponding to an asymptotic period-doubling loss of stability. This completes the proof of item~(c).
\end{proof}

\begin{remark}
The set \(\widetilde{\mathcal B}_{\mathrm{st}}\) is obtained from the large-amplitude asymptotics of the Schur--Cohn conditions, and therefore describes the boundary of the stability region only asymptotically. The open set \(\mathcal B_{\mathrm{st}}\) is obtained by continuation of the stable symmetric branch and may contain finite- and small-amplitude cycles, as confirmed by the numerical catalog.
\end{remark}

\begin{corollary}[Perturbative inclusion in the stability band]
	\label{cor:perturbative_band}
	Assume that the stability band can be written as
	\[
	\mathcal B_{\mathrm{st}}
	=
	\{(C,H)\,;\; C\in I,\; H_{\min}(C)<H<H_{\mathrm{crit}}(C)\},
	\]
	where \(I\subset\mathbb R\) is the admissible range of \(C\), and define
	\[
	\Delta(C):=H_{\mathrm{crit}}(C)-H_{\min}(C)>0.
	\]
	Then, for every fixed \(C\in I\) and every perturbation
	\[
	0<\delta<\Delta(C),
	\]
	the parameter choice
	\[
	H=H_{\mathrm{crit}}(C)-\delta
	\]
	belongs to \(\mathcal B_{\mathrm{st}}\). Consequently, the Weierstrass Preparation Theorem closing relation yields a crossing symmetric limit cycle, and this cycle is orbitally asymptotically stable.
\end{corollary}

\begin{proof}
	Fix \(C\in I\). If \(0<\delta<\Delta(C)=H_{\mathrm{crit}}(C)-H_{\min}(C)\), then
	\[
	H_{\min}(C)<H_{\mathrm{crit}}(C)-\delta<H_{\mathrm{crit}}(C),
	\]
	hence \((C,H_{\mathrm{crit}}(C)-\delta)\in\mathcal B_{\mathrm{st}}\). The conclusion follows from Theorem~\ref{teo:stability-band}.
\end{proof}

Finally, the geometry of the lower boundary can be expressed in terms of the hyperbola \(\Gamma_1\). More precisely, \(D(\Gamma_1)\) denotes the discriminant of the hyperbola associated with the quadratic form defining \(\Gamma_1\). In particular, the condition \(D(\Gamma_1)>0\) guarantees that \(\Gamma_1\) is indeed a nondegenerate hyperbola, and the asymptotic transition at the lower boundary of \(\mathcal B_{\mathrm{st}}\) is governed by this hyperbolic geometry.

\section*{Funding}

No funding was received to assist with the preparation of this manuscript.

\section*{Data availability}
Data sharing not applicable to this article as no datasets were generated or analyzed during the current study.

\section*{Declaration of interest}

The authors declare that they have no conflict of interest.

\section*{Author contributions}

S. Ferreira led the theoretical development of the paper, carried out most of the analytical calculations, prepared most of the figures, and wrote the manuscript. B. Freitas contributed to the correction of equations and conceptual points, participated in the critical review of the manuscript, and added figures. J. Medrado contributed to the correction of equations and conceptual points and to the critical review of the manuscript.  All authors have read and approved the final version of the manuscript.

\section*{References}
\small
\bibliographystyle{elsarticle-num} 
\bibliography{referencial}

\end{document}